%% file: thes-full.tex
\documentclass[11pt]{report}
\usepackage{csthesis}
\usepackage{amsmath, amsthm, amssymb, graphicx,caption,subcaption}
\usepackage{makeidx}  
\makeindex

\graphicspath{{./Pictures/}}

\newtheorem{theorem}{Theorem}
\newtheorem{corollary}[theorem]{Corollary}
\newtheorem{prop}[theorem]{Proposition}
\newtheorem{lem}[theorem]{Lemma}
\newtheorem{conjecture}{Conjecture}
\newtheorem{question}[conjecture]{Question}
\newtheorem*{Menger}{Menger's Theorem}
\newtheorem*{RS}{Robertson-Seymour Theorem}

\theoremstyle{definition}
\newtheorem*{example}{Example}
\newtheorem*{definition}{Definition}

\newtheorem{remnum}[theorem]{Remark}

\newtheorem{case}{Case}

\newcommand{\startingcases}{\setcounter{case}{0}}
\newcommand{\exitingcases}{\vspace{3 mm}}

\def\mone{_1M}
\def\mtwo{_2M}

\def\cut{\setminus}
\def\contract{/}

\makeatletter
\newcommand\otherlist{%
    \addcontentsline{toc}{chapter}{\otherlistname}
    \if@twocolumn
      \@restonecoltrue\onecolumn
    \else
      \@restonecolfalse
    \fi
    \chapter*{\otherlistname
      \@mkboth{\MakeUppercase\otherlistname}%
              {\MakeUppercase\otherlistname}}%
    \@starttoc{lop}%
    \if@restonecol\twocolumn\fi
    }
\newcommand*\l@program{\@dottedtocline{1}{1.5em}{2.3em}}
\newcommand\otherlistname{List of Programs}
\newcommand\programname{Program}
\newcounter{program}[chapter]
\renewcommand\theprogram{\thechapter.\@arabic\c@program}
\def\fps@program{tbp}
\def\ftype@program{3}
\def\ext@program{lop}
\def\fnum@program{\programname~\theprogram}

\newenvironment{program*}
               {\@dblfloat{program}}
               {\end@dblfloat}
\makeatother

\begin{document}
\figurespagetrue
\tablespagetrue
\otherlistpagetrue

\include{files/titapp} 
\include{files/abstract} 

\lists


\beforetext

\include{files/introduction}

\include{files/basics}

\include{files/delta-y}

\include{files/connectivitynew}


\include{files/trisubnew}


\include{files/last}


\include{files/bibl}

\include{files/ind}

\end{document}

%% file: files/titapp.tex
%
%


\title{Forbidden Minors for 3-Connected Graphs With No Non-Splitting 5-Configurations}
\author{Iain Crump}
\qualification{B.Sc., University of Winnipeg, 2009.}
\submitdate{Summer 2012}
\copyrightyear{2012}


\chair{Dr.~Marni Mishna,\\
	Associate Professor, Mathematics\\
	Simon Fraser University} 
\signatory{Dr.~Karen Yeats,\\
       Assistant Professor, Mathematics\\
       Simon Fraser University\\ 
        Senior Supervisor}
\signatory{Dr.~Matthew DeVos,\\
       Assistant Professor, Mathematics\\
       Simon Fraser University\\ 
       Committe Member}
\signatory{Dr.~Luis Goddyn,\\
       Professor, Mathematics\\
       Simon Fraser University\\
       SFU Examiner}

\beforepreface

%% file: files/abstract.tex
%
%

\prefacesection{Abstract}

For a set of five edges, a graph splits if one of the associated Dodgson polynomials is equal to zero. A graph $G$ splitting for every set of five edges is a minor-closed property. As such there is a finite set of forbidden minors $\mathcal{F}$ such that if a graph $H$ does not contain a minor isomorphic to any graph in $\mathcal{F}$, then $H$ splits. In this paper we prove that if a graph $G$ is simple, 3-connected, and splits, then $G$ must not contain any minors isomorphic to $K_5$, $K_{3,3}$, the octahedron, the cube, or a graph that is a single $\Delta$-Y transformation away from the cube. As such this is the set of all simple 3-connected forbidden minors. The complete set of 2-connected or non-simple forbidden minors remains unresolved, though a number have been found.

%% file: files/introduction.tex
%
%

\chapter{Introduction}
\label{introduction}

\section{Background}

Feynman diagrams arise in physics as a way of understanding complicated interactions between elementary particles and the calculations that arise in computing changes in energy and the probability of these interactions occurring (\cite{fieldtheory}). Predictions using renormalized Feynman diagrams are known to have high precision in perturbative quantum field theory (\cite{hopf}). The integrals used for these calculations quickly become difficult to calculate, but much of the number theoretic content of massless quantum field theory is contained in the residues (\cite{monster}). As such, it is useful to be able to extract the residues from these integrals.

In particular, we may treat Feynman diagrams as graphs. For a Feynman diagram $G$, to each internal edge $e \in E(G)$ we associate a variable $\alpha_e$. Then, the Kirchhoff polynomial for $G$, introduced in \cite{kirchhoff}, is $$\Psi_G = \sum_{\substack{T \text{ a spanning} \\ \text{tree of } G}} \hspace{2mm} \left(\prod_{e \notin E(T)} \alpha_e\right).$$ The residue of the Feynman integral for this Feynman diagram $G$ in massless scalar field theory is $$I_G = \int\limits_0^\infty \cdots \int\limits_0^\infty \frac{\prod_{i=1}^{n-1} \mathrm{d}\alpha_i}{\Psi^2_G} \big|_{\alpha_n = 1}.$$ This integral converges for primitive divergent graphs (\cite{monster}, \cite{c2}).

These integral calculations are difficult and in many cases require deep analytic and numeric methods (\cite{c2}), but it is proven in \cite{monster} that, for a particular class of graphs, the fifth stage of integration produces a recognizable denominator, the five-invariant, which arises from the Kirchhoff polynomial in a natural way. Specifically, we may calculate the five-invariant using modified Kirchhoff polynomials, known as Dodgson polynomials. Like the Kirchhoff polynomial, Dodgson polynomials are linear in each Schwinger coordinate. The five-invariant is the difference of products of two Dodgson polynomials. If the five-invariant, when fully factored, is linear in each factor in at least one edge variable, we may easily calculate the sixth denominator, and even the sixth partial integral. It is of particular interest, then, when one of the Dodgson polynomials used in calculating the five-invariant is equal to zero, as this will guarantee that the five-invariant can be factored into terms linear in all Schwinger coordinates.

Certain graph obstructions, however, prevent using this method to calculate this fifth stage denominator for any choice of five edges. The five-invariant being able to be factored into linear terms for all Schwinger coordinates for any choice of five edges is a minor-closed property, though. As such, there exists a finite set of graph obstructions.

The central goal of this thesis is the characterization of the forbidden minors as they may appear in 3-connected simple graphs. Hence, we show that a 3-connected simple graph that is free of five particular minors must factor as desired. Chapter \ref{introduction} will provide background necessary to understanding the work. We will introduce key concepts related to Dodgson polynomials as denominators of partial Feynman integrals. In particular, we will introduce standard theorems that allow for a graph theoretic approach to this problem.

Chapter \ref{basics} will move towards considering Kirchhoff polynomials purely in graph theoretic terms. By this point, the tools will be in place to prove all theorems using only trees that span two particular minors associated with each Dodgson polynomial for a particular graph. Here, we will introduce our first restrictions on minor-minimal non-splitting graphs, and more general methods for demonstrating that a non-splitting graph is not minor-minimal. We also introduce our first minor-minimal non-splitting graphs.

Chapter \ref{delta-y} introduces the complete set of minor-minimal graph obstructions that have been found so far. In particular, families of forbidden graphs arise by a natural graph operation, the $\Delta$-Y transformation. This operation is explored, as well as its effects on Dodgson polynomials and non-splitting 5-configurations.

Chapter \ref{cuts} considers graph connectivity. In this chapter, we show that cut vertices cannot appear in minor-minimal obstructions. Further, theorems are cited that prove minor-minimal obstructions cannot be 5-connected. Hence, minor-minimal non-splitting graphs can be two-, three-, or four-vertex connected. There are, however, specific restrictions on graphs that are two-vertex connected. Further, there are precisely two four-vertex connected graphs.

In Chapter \ref{trisub}, we restrict ourselves to three-vertex connected simple graphs only. These restrictions arises naturally from our interest in the denominators of Feynman integrals, as graphs with multiple edges or two vertex cuts are trivial in the theory to which our applications apply. In this chapter, we prove that we have constructed a complete set of obstructions for graphs that are three-vertex connected and simple.

Finally, Chapter \ref{last} concludes the thesis. We consider the result in the framework of applications to computing residues of partial Feynman integrals, and put forward a number of conjectures that arise throughout the thesis.

I would like to thank Samson Black for his contribution of well-documented Sage programs that allowed for a large number of enormous calculations that were necessary to complete this thesis. Further, he worked with us though a large number of ideas that eventually led to the solution presented within.

\subsection{Feynman Graphs}

The goal of \emph{quantum field theory} is to understand the behaviour and interactions of \emph{elementary particles} (\cite{fieldtheory}, \cite{quantum}). These are the fundamentally indivisible particles of which the universe is composed. The founders of this field were hoping to find a unified description of the elementary particles and the way in which they interact. Calculations in quantum field theory are used to describe the energy changes associated with these interactions, and probability of their occurrence. The calculations alone, though, are abstract and difficult to visualize.

\emph{Feynman diagrams}, introduced in 1948 by Richard Feynman, offer a visual model to approach the interactions and calculations in quantum field theory. Feynman diagrams, to mathematicians, are multigraphs on a specialized set of edges which may further contain parallel edges, loops, and both directed and undirected edges. They are additionally allowed to contain \emph{external edges}; these are edges incident only with a single vertex used to represent particles entering or exiting the represented system. Edges that are not external are \emph{internal edges}.

For quantum field theorists, a particle can be described by a set of attributes such as position in spacetime, spin, charge, mass, and so on. We treat particles as an indexed list of these attributes. Then, Feynman diagrams construct edges out of two half-edges. Each half edge is associated with a particle, and the joining of two half edges represents a change in at least one attribute of that particle. If a particle exhibits no change to any attribute (and note that this could simply be a change in spacetime position), it is of no interest. As such, vertices in Feynman diagrams are always of degree three or more, and demonstrate the disintegration or combining of particles.

\begin{example} Suppose the indices in the Feynman diagram in Figure \ref{feynmanexample} each contain the collection of attributes that define a particular particle. Then, this diagram represents the probability that a particle with attributes $a$ will disintegrate into particles $b$ and $d$, particles $b$ and $d$ will change to particles with attributes $c$ and $e$, respectively, and these two particles will combine to form particle $f$. \end{example}

\begin{figure}[h]
  \centering
    \includegraphics[scale=1.0]{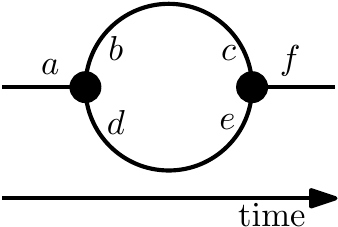}
  \caption{An example of a Feynman diagram.}
\label{feynmanexample}
\end{figure}

In particular, note the external edges associated to particles $a$ and $f$ in Figure \ref{feynmanexample}. We assume that these particles were simply in the system, and as such are only of interest when disintegrating or reforming.

Time axes may be included in diagrams to indicate in which direction one reads the interactions, though they are not necessary. Reading a diagram in another direction simply describes a different process. Counterintuitively, in field theory the probabilities associated with any interaction at a vertex is equal regardless of the direction from which one approaches it. Thus, calculations of probabilities and energies associated to a diagram are the same no matter which way one reads the diagram, and so all such processes can be considered together.

Figure \ref{feynmanreal} describes the disintegration of a photon, $\gamma$, into an electron and positron, $e^-$ and $e^+$, respectively. The electron and positron later reform as a photon. Standard notation from quantum electrodynamics is used, in which a photon is represented by a wavy line, and electrons and positrons are associated to directed edges, electrons associated to edges that move with the flow of time and positrons associated to edges moving against it.

\begin{figure}[h]
  \centering
    \includegraphics[scale=1.0]{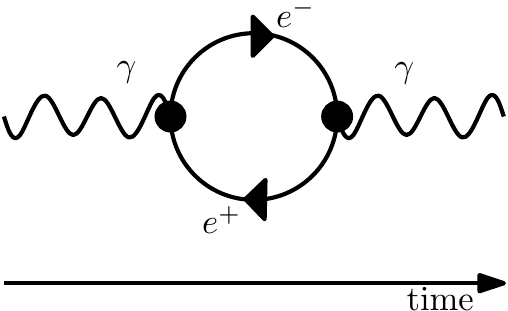}
  \caption{A Feynman diagram describing a real-world particle interaction.}
\label{feynmanreal}
\end{figure}

\subsection{Feynman Integrals}

For an arbitrary Feynman diagram, the \emph{Feynman integral} is used to calculate the probability amplitudes associated with the interactions in the diagram. The following definitions are necessary for our understanding of Feynman integrals.

\begin{definition} Let $G$ be a graph and $H$ a subgraph of $G$. We say that $H$ is a \emph{spanning subgraph} if $V(H) = V(G)$. A \emph{tree} is a connected graph with no cycles. It follows that a \emph{spanning tree} is a spanning subgraph that is a tree. We define $T_G$ to be the set of all spanning trees of $G$. \end{definition}

\begin{definition} Let $G$ be a graph. For each edge $e \in E(G)$, we associate a variable $\alpha_e$, known as the \emph{Schwinger coordinate}. The \emph{Kirchhoff polynomial} for $G$ is defined as \begin{equation} \Psi_G = \sum_{T \in T_G} \left( \prod_{e \notin T} \alpha_e \right). \end{equation} \end{definition}

Whence, for a Feynman diagram $G$, index the internal edges in $E(G)$ numerically for convenience, $1, 2, ..., |E(G)|$. Using notational conventions from \cite{quantum}, the scalar Feynman integral is $$\int\limits_0^\infty \cdots \int\limits_0^\infty  \frac{e^{iQ_G(\alpha)}}{\left( i(4 \pi)^2 \right)^{|E(G)|+1-|V(G)|} \left( \Psi_G \right)^2}  \prod_{k=1}^{|E(G)|} \left( e^{-i \alpha_k m_k^2} \mathrm{d}\alpha_k \right). $$ The $m_k$ values are the masses associated to the particles. The $Q_G(\alpha)$ seen in the numerator is a function over the Schwinger coordinates that is similar to the Kirchhoff polynomial, based instead on edges that are in edge cut sets that produce precisely two connected components and incorporating the external momenta. Note that this is a scalar Feynman integral, which is simpler than the general case. Bluntly, equations of this sort do not attract potential physicists or mathematicians.

There are numerous cases in which the Feynman integral will seem to be equal to infinity, which as a measure of energy is clearly impossible. This is due to particle self-interactions. The concept of \emph{renormalization} was introduced to deal with this. Put simply, this subtracts infinity from the result (following specific rules) to produce finite values when calculating the Feynman integral. The rules, when first introduced, lacked firm mathematical underpinnings, though. Dirac, a central figure in quantum field theory, criticized the seemingly arbitrary method in which inconvenient infinities were dismissed (see \cite{dirac}, page 184). The renormalization Hopf algebra, introduced as an approach to Feynman graphs by Dirk Kreimer (\cite{renorm}, \cite{hopf}), provides a mathematical backing to renormalization.

To simplify the calculation while retaining key number theoretic information we now set the external parameters and masses to zero and remove the overall divergence by setting $\alpha_n =1$ (see \cite{quantum}, pages 294-299, for the calculations involved). In doing so, we simplify the Feynman integral to its residue, $$I_G = \int\limits_0^\infty \cdots \int\limits_0^\infty \frac{\prod_{i=1}^{n-1} \mathrm{d}\alpha_i}{\Psi^2_G} \big|_{\alpha_n = 1},$$ which is much more manageable. This simplification further preserves much of the content of the original integral. This integral converges for primitive divergent graphs. From now on, Feynman integrals will be considered only in this context.

One method for calculating Feynman integrals uses \emph{partial Feynman integrals}. Specifically, we apply an ordering to the internal edges of a Feynman diagram and integrate with regard to the Schwinger coordinate of each edge in order. The $i^\text{th}$ partial Feynman integral of a Feynman diagram $G$ is denoted $I_i$.

\section{Five-Invariants and Dodgsons}
\label{1.1}

We now begin the move to a more graph theoretic approach to this material. Many general graph theory definitions will be assumed as known. Notational conventions will be as found in \cite{graph}. For the purposes of this paper, all graphs are assumed to be undirected multigraphs with loops allowed, unless otherwise stated.

Despite the fact that this work is derived from Feynman graphs, we will no longer consider graphs with external edges. This is because the Feynman integral in the simplified form that we care about does not involve external edges.

On occasion, a portion of the graph may be unimportant, but the relative position it maintains must be specifically noted. We may represent this portion of the graph in these instances with a grey blob showing where the particular part of the graph is positioned, but containing no internal information. An example of this is shown in Figure \ref{blob}. In particular, note that edge $e$ is considered to be a part of the blob. This is not necessary; $e$ could also be left outside the blob. In any case of potential ambiguity, edges of this sort will be considered specifically if necessary.

\begin{figure}[h]
  \centering
    \includegraphics[scale=1.0]{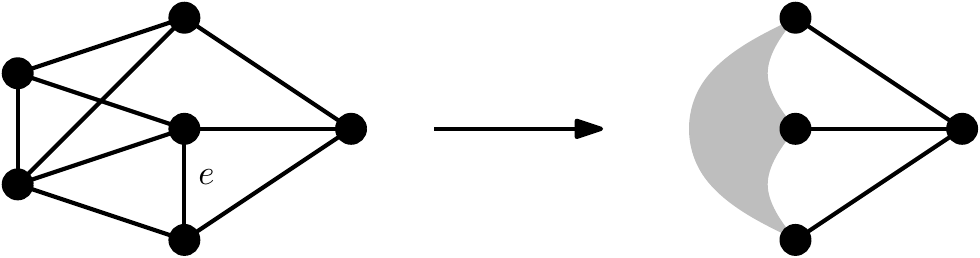}
  \caption{A blob diagram for a particular graph.}
\label{blob}
\end{figure}

Recall from the previous section that the Kirchhoff polynomial is defined to be $\Psi_G = \sum_{T \in T_G} \left( \prod_{e \notin T} \alpha_e \right).$ The following is an example of such a calculation.

\begin{example} Consider the graph in Figure \ref{introexample}, call it $G$. Any spanning tree of $G$ must have precisely three edges. If we are to include edge $e_5$ in the spanning tree, we must include precisely one of $e_1$ or $e_2$ and one of $e_3$ or $e_4$. If edge $e_5$ is not in the spanning tree, we may include any three edges from the cycle induced by edges $e_1, ... , e_4$. Then, if we construct spanning trees as induced by edge sets, \begin{align*} T_G = \{ \{e_1,e_3,e_5 \},&\{e_1,e_4,e_5 \},\{e_2,e_3,e_5 \}, \{e_2,e_4,e_5 \}, \\ &\{e_1,e_2,e_3 \}, \{e_1,e_2,e_4 \},\{e_1,e_3,e_4 \},\{e_2,e_3,e_4 \}  \}. \end{align*} It follows that $$\Psi_G = \alpha_{e_2}\alpha_{e_4} + \alpha_{e_2}\alpha_{e_3} + \alpha_{e_1}\alpha_{e_4} + \alpha_{e_1}\alpha_{e_3} + \alpha_{e_4}\alpha_{e_5} + \alpha_{e_3}\alpha_{e_5} + \alpha_{e_2}\alpha_{e_5} + \alpha_{e_1}\alpha_{e_5}  .$$\end{example}

\begin{figure}[h]
  \centering
    \includegraphics[scale=1.0]{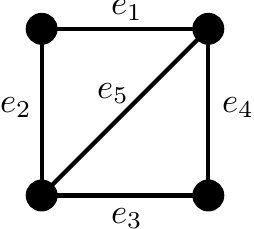}
  \caption{A tiny little graph.}
\label{introexample}
\end{figure}

In the case that the graph $G$ is a tree, we define $\Psi_G = 1$. If $G$ is not connected, then trivially $T_G = \emptyset$ and we say $\Psi_G = 0$. As disconnected graphs are trivial, we will assume that all graphs considered are connected.

For an undirected graph $G$, apply an arbitrary orientation to the edges, and create an incidence matrix for this new directed graph, $\xi_G$. Specifically, $\xi_G$ is a $|E(G)| \times |V(G)|$ matrix such that $$ (\xi_G)_{e,v} =\begin{cases} 1 & \text{if } v \text{ is the source vertex of edge } e, \\ -1 & \text{if } v \text{ is the target vertex of edge } e, \\ 0 & \text{otherwise.} \end{cases} $$ This is the transpose of the more standard definition of the incidence matrix. Note that $\xi_G$ is not well defined, as it depends on the arbitrary orientation chosen for the edges and ordering of the vertices. Let $A$ be the diagonal matrix with entries $\alpha_e$ for $e \in E(G)$. Let $\widetilde{M}_G$ be the block matrix constructed as follows; \begin{equation*} \widetilde{M}_G = \left[\begin{array}{c|c} A & \xi_G \\ \hline -\xi^T_G & \bf{0}\end{array}\right].\end{equation*} The first $|E(G)|$ rows and columns are indexed by the edges of $G$, and the remaining $|V(G)|$ rows and columns are indexed by the set of vertices of G, in some arbitrary order.

Let $\widehat{\xi}_G$ be a submatrix of the incidence matrix $\xi_G$ obtained by deleting an arbitrary column. We define matrix $M_G$ as, \begin{equation*} M_G = \left[\begin{array}{c|c}A & \widehat{\xi}_G \\ \hline -\widehat{\xi}^T_G & \bf{0}\end{array}\right].\end{equation*} For the reasons stated prior this matrix is not well defined, and further column deleted in creating $\widehat{\xi}_G$ was arbitrary.

\begin{example} Applying an arbitrary orientation to the graph in Figure \ref{introexample} we get the graph seen in Figure \ref{introexampledirected}. 

\begin{figure}[h]
  \centering
    \includegraphics[scale=1.0]{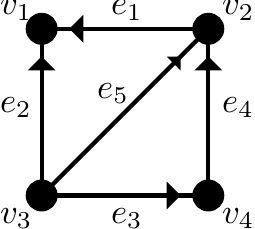}
  \caption{A tiny little directed graph.}
\label{introexampledirected}
\end{figure}

For this graph, one possible incidence matrix is \begin{equation*} \xi_G = \left[\begin{array}{c c c c}
-1 & 1 & 0 & 0 \\
-1 & 0 & 1 & 0 \\
0 & -1 & 0 & 1 \\
0 & 0 & 1 & -1 \\ 
0 & -1 & 1 & 0 \end{array}\right]. \end{equation*} With this incidence matrix, one possible matrix $M_G$ is \begin{equation*} M_G = \left[ \begin{array}{c c c c c|c c c} 
\alpha_{e_1} &  &  &  &  & -1 & 1 & 0 \\
 & \alpha_{e_2} &  &  &  & -1 & 0 & 1 \\
 &  & \alpha_{e_3} &  &  & 0 & -1 & 0 \\
 &  &  & \alpha_{e_4} &  & 0 & 0 & 1 \\
 &  &  &  & \alpha_{e_5} & 0 & -1 & 1 \\
\hline
1 & 1 & 0 & 0 & 0 &  &  &  \\
-1 & 0 & 1 & 0 & -1 &  &  &  \\
0 & -1 & 0 & -1 & -1 &  &  &   \end{array} \right]. \end{equation*} For the sake of clarity, zeroes in the upper left and lower right blocks have been omitted. Note that the column in $\xi_G$ corresponding to vertex $v_4$ was deleted in this construction. \end{example}

The following Lemma is due to Kirchhoff, and appears in \cite{kirchhoff}. It also appears as Lemma 20 in \cite{monster}.

\begin{lem} \label{kirk} Let $G$ be a connected graph. Let $I \subset E(G)$ such that $|I| = |E(G)| - |V(G)| +1$. With the incidence matrix $\widehat{\xi}_G$ as defined previously, create matrix $\widehat{\xi}_{G}(I)$ by deleting every row indexed by edges in $I$. Then, $$\det (\widehat{\xi}_{G}(I)) = \begin{cases} \pm 1, & \mbox{if edge set } E(G) - I\mbox{ induces a spanning tree of } G \\ 0, & \mbox{otherwise} \end{cases}.$$ \end{lem}

\begin{theorem} \label{determinant} For an arbitrary graph $G$, the Kirchhoff polynomial $\Psi_G = \det(M_G)$. Specifically, the determinant does not depend on the edge orientation of $G$, nor the column deleted from $\xi_G$. \end{theorem}

\begin{proof} Let $|E(G)| = N$ and $|V(G)|=v$. Note that $A$ is a diagonal square matrix with Schwinger coordinates as diagonal entries. It follows that $A$ is invertible. Using a modified version of the Schur complement, \begin{align*} \det(M_G) &= \alpha_1 \cdots \alpha_N \det( \mathbf{0} - (-\widehat{\xi}_G^T A^{-1}\widehat{\xi}_G) ) \\ &= \alpha_1 \cdots \alpha_N \det( \widehat{\xi}_G^T A^{-1}\widehat{\xi}_G ) \end{align*}

For an $m \times n$ matrix $A$ and set $S$, let $A_{m,S}$ be the matrix created by deleting all columns of $A$ indexed by numbers not in $S$ and let $A_{S,n}$ be the matrix created by deleting all the rows of $A$ indexed by numbers not in $S$. By the Cauchy-Binet formula and indexing rows and columns of our matrix by the edges of $G$, \begin{align*} \det(\widehat{\xi}_G^T A^{-1}\widehat{\xi}_G) &= \sum_{\substack{|S|=v-1 \\ S \subset E(G)}} \det((\widehat{\xi}_G^T)_{v-1,S}) \det((A^{-1}\widehat{\xi}_G)_{S,v-1}) \\ &= \sum_{\substack{|S|=v-1 \\ S \subset E(G)}} \left( \prod_{e \in S} \frac{1}{\alpha_e} \right) (\det(\widehat{\xi}_G)_{S,v-1})^2. \end{align*} By Lemma \ref{kirk}, the determinant $\det(\widehat{\xi}_G)_{S,v-1}$ is equal to $\pm 1$ (and hence the square is equal to one) if and only if $S$ is a spanning tree of our graph $G$. Hence, $\alpha_1 \cdots \alpha_N \det( \widehat{\xi}_G^T A^{-1}\widehat{\xi}_G )$ is the sum over all spanning trees of the product of edges not in the tree. \end{proof}

\begin{definition} For an edge $e \in E(G)$, we define the \emph{edge deletion}, $G \cut e$, to be the graph created by removing edge $e$. We define the \emph{edge contraction}, $G \contract e$, to be the graph created by removing edge $e$ and identifying the ends of $e$. For a set $S \subseteq E(G)$, define $G \cut S$ to be the graph created by deleting all edges $e \in S$, and similarly $G \contract S$ to be the graph created by contracting all edges $e \in S$. From a graphical perspective, contracting a loop is the same as deleting a loop, though the operations must be treated seperately for our purposes. \end{definition}

The following proposition follows logically.

\begin{prop} \label{smallcutcontract} For a graph $G$ and $e \in E(G)$, $\Psi_G = \alpha_e \Psi_{G \cut e} + \Psi_{G \contract e}$. \end{prop}

The proof of this proposition is reasonably straightforward, but omitted here for the sake of brevity. It does follow immediately from Theorem \ref{cutablecontractablegeneral} in Chapter \ref{basics} and can also be found as Lemma 14 in \cite{monster}.

\begin{definition} Let $I,J,K \subseteq E(G)$. We further define the matrix $M_G(I,J)_K$ by deleting rows indexed by the edges in $I$, columns indexed by the edges in $J$, and setting $\alpha_e = 0$ for all $e \in K$ in $M_G$. Forcing $|I| = |J|$, we may define $\Psi^{I,J}_K = \det(M_G(I,J)_K)$, and this definition is well-defined up to sign, the result of the orientation on $G$ in creating a digraph and the row and column deleted in our original matrix being arbitrary. We call $\Psi^{I,J}_K$ the \emph{Dodgson polynomial} corresponding to edge sets $I$, $J$, and $K$. If $K = \emptyset$, we write this as $\Psi^{I,J}$. \end{definition}

\begin{prop} \label{calcdodgson} For a graph $G$ and $I,J,K \subseteq E(G)$ such that $|I| = |J|$, the Dodgson polynomial $$\Psi^{I,J}_K = \sum_{T \subseteq E(G)} \left( \pm \prod_{e \notin T \cup S} \alpha_e \right)$$ where $S = I \cup J \cup K$ and the sum is over all edge sets which induce spanning trees in both graph minors $G \cut I \contract ((J - (I \cap J)) \cup (K - (I \cap K)))$ and $G \cut J \contract ((I - ( J \cap I)) \cup (K- (J \cap K)))$. \end{prop}

\begin{proof} Let $G' = G \cut (I \cap J) \contract K$. Using this minor, we may assume that $I \cap J = K = \emptyset$. As in Theorem \ref{determinant}, \begin{align*} \det(M_{G'}(I,J)) &= \sum_{U \subseteq G' \cut (I \cup J)} \prod_{u \notin U} \alpha_u \det \left( \begin{array}{c|c}
\bf{0} & \xi_{G'}(U \cup I) \\ \hline -\xi^T_{G'} (U \cup J) & \bf{0} \end{array}\right) \\ &= \sum_{U \subseteq G' \cut (I \cup J)} \prod_{u \notin U} \alpha_u \det( \xi_{G'} (U \cup I) ) \det (\xi_{G'} (U \cup J) ) .  \end{align*} It follows from Lemma \ref{kirk} that for both $\det(\xi_{G'}(U \cup I))$ and $\det(\xi_{G'}(U \cup J))$ to be non-zero, $U \cup I$ and $U \cup J$ must induce spanning trees of $G'$. By passing to the minor $G'$, we have assumed that $I \cap J = \emptyset$, and as such $U \cup I$ has no edges in $J$. It follows that $U$ is a spanning tree in $G' \cut I \contract J$, and similarly $U$ is a spanning tree in $G' \cut J \contract I$. For any such edge set $U$, Lemma \ref{kirk} implies that these determinants are equal to $\pm 1$, and as such their product is $\pm 1$. \end{proof}

It is important to note that there can be no cancellation between terms when calculating Dodgson polynomials using this method. Specifically,  cancellation could only occur if the same edge set that produces spanning trees in both minors was considered twice.

\begin{example} Consider the graph $G$ in Figure \ref{intdodgson}. In calculating Dodgson $\Psi^{e_5,e_4}$ using the method in Proposition \ref{calcdodgson}, we look for edge sets that induce spanning trees in both $G \cut e_5 \contract e_5$ and $G \cut e_4 \contract e_5$, both minors also shown in this figure. 

Edge sets $\{e_1,e_3\}$ and $\{e_2,e_3\}$ form the only common spanning trees. Therefore, $\Psi^{e_5,e_4} = \pm \alpha_{e_2} \pm \alpha_{e_1}$.

\begin{figure}[h]
  \centering
    \includegraphics[scale=1.0]{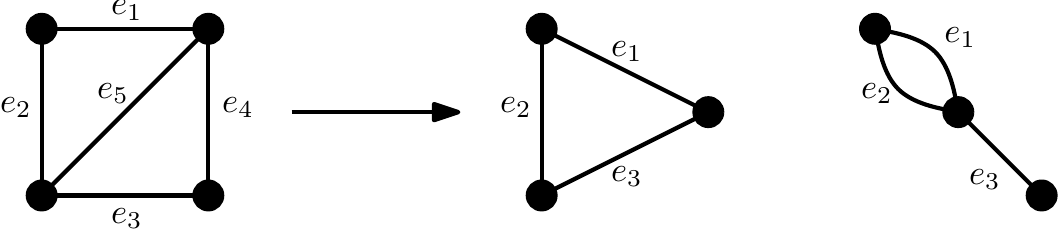}
  \caption{A graph and the minors used in calculating a Dodgson polynomial.}
\label{intdodgson}
\end{figure}

 \end{example}

\begin{remnum} \label{alwayscut} Let $G$ be a graph. Assume that $I \cup J$ and $K$ are disjoint. Any edge $e \in K$ is contracted in both of the minors of $G$ used in the method for calculating $\Psi_K^{I,J}$ given in Proposition \ref{calcdodgson}. Further, any edge $e \in I \cap J$ will be deleted in both minors. Our particular interests lie in Dodgson polynomials such that one of $|K|$ or $|I \cap J|$ is equal to one, the other zero. As such, there will always be precisely one edge that is deleted or contracted in both minors related to a particular Dodgson. \end{remnum}

\begin{definition} For a graph $G$, a \emph{5-configuration} is a set $S \subseteq E(G)$ such that $|S| = 5$. \end{definition}

The Dodgsons we concern ourselves with arise from 5-configurations. Specifically, $|I \cup J \cup K| =5$, and we are interested in edge sets such that $|I \cap K| = |J \cap K| = 0$, and either either $|I| = |J| = 2$, $|K| = 1$, and $|I \cap J| = 0$ or $|I| = |J| = 3$, $|K| = 0$, and $|I \cap J| = 1$. Trivially, $\Psi^{I,J}_K =\Psi^{J,I}_K$. As such, for a fixed 5-configuration, there are thirty generically distinct Dodgsons of these forms associated with that set of five edges to be considered. For notational convenience we will write $\Psi^{e_1e_2,e_3e_4}_{e_5}$ in place of $\Psi^{ \{e_1,e_2\}, \{e_3,e_4\}}_{\{e_5\}}$, as an example.

\begin{definition} Let $G$ be a graph and fix a 5-configuration $S \subseteq E(G)$. Define $D^G_S$ to be the set of the thirty distinct Dodgsons associated with the 5-configuration $S$. If the graph $G$ is clear from context, we will write this as $D_S$. \end{definition}

For a graph $G$ and a 5-configuration $S$, let $\mone$ and $\mtwo$ be the two minors created as in Proposition \ref{calcdodgson} in calculating a fixed Dodgson $D \in D_S$. Note that the order of the indices is arbitrary. Treating spanning trees as edge sets, we will say that a tree spans both $\mone$ and $\mtwo$ if the edge set creates a tree in both minors. In instances where we need to specify how edges are deleted or contracted in a particular minor, say $\mone$, we will write $\mone^{\cut A}_{\contract B}$ for edge sets $A,B \subseteq S$. Note that we may not include all the edges in $S$, only the edges of particular interest.

\begin{example} Fix an arbitrary graph $G$ and edge set $S = \{e_1,...,e_5\} \subseteq E(G)$. In calculating the Dodgson $\Psi^{e_1e_2,e_3e_4}_{e_5}$, we take two minors $\mone^{\cut e_1e_2}_{\contract e_3e_4e_5}$ and $\mtwo^{\cut e_3e_4}_{\contract e_1e_2e_5}$. If we need focus only on edges $e_1$ and $e_5$, we may instead write these as $\mone^{\cut e_1}_{\contract e_5}$ and $\mtwo_{\contract  e_1e_5}$. \end{example}

\begin{definition} Let $G$ be a graph, $S$ a 5-configuration, and $D \in D_S$ with minors $\mone$ and $\mtwo$. For $D$, let $T_D$ be the set of all edge sets $E\subseteq E(G) -S$ such that $E$ induces spanning trees of both minors $\mone$ and $\mtwo$ associated with Dodgson $D$. We will refer to these edge sets as trees when dealing with minors $\mone$ and $\mtwo$, as their key property is inducing spanning trees in both $\mone$ and $\mtwo$. \end{definition}

By Proposition \ref{calcdodgson}, for 5-configuration $S$ and $D \in D_S$, $$D = \sum_{T \in T_D} \left( \pm \prod_{e \notin T} \alpha_e \right).$$ With this and the definition of the Dodgson, we have two distinct methods of determining if a Dodgson is equal to zero. Our proofs will always rely on spanning trees. Specifically, we are interested in cases in which the Dodgson is equal to zero, and hence in a Dodgsons $D$ such that there is no $T \in T_D$. As such, the ambiguity in sign in each term using spanning trees is not an issue. There are calculations performed in this thesis that produce non-zero Dodgsons. These were all calculated using the matrix method, specifically using Sage code produced by Samson Black.

\begin{definition} Let $G$ be a graph and fix five edges $S = e_1, ..., e_5 \in E(G)$. As in the definition of a Dodgson, fix a matrix $M_G$ to use in the calculation of all Dodgsons $D \in D_S$. The \emph{five-invariant} is the polynomial $$^5\Psi(e_1,e_2,e_3,e_4,e_5) = \pm (\Psi^{e_1e_2,e_3e_4}_{e_5} \Psi^{e_1e_3e_5,e_2e_4e_5} - \Psi^{e_1e_3,e_2e_4}_{e_5} \Psi^{e_1e_2e_5,e_3e_4e_5}).$$ \end{definition}

The 5-invariant was first implicitly found in \cite{polys}, equation (8.13).

The following is Lemma 87 in \cite{monster}.

\begin{lem} \label{reindex} Reordering the edges in a five-invariant may at most change the sign of the polynomial. \end{lem}

While the general proof is more involved, some specific cases are easily proven. It was noted earlier that $\Psi_K^{I,J} = \Psi_K^{J,I}$. Hence, $^5\Psi(e_1,e_2,e_3,e_4,e_5) = ^5\Psi(e_3,e_4,e_1,e_2,e_5)$. Exchanging edges $e_2$ and $e_3$ gives $^5\Psi(e_1,e_3,e_2,e_4,e_5) = - ^5\Psi(e_1,e_2,e_3,e_4,e_5)$, which can easily be verified by expanding $^5\Psi(e_1,e_3,e_2,e_4,e_5)$ in terms of Dodgsons. We may similarly show that $^5\Psi(e_4,e_2,e_3,e_1,e_5) = - ^5\Psi(e_1,e_2,e_3,e_4,e_5)$.

\begin{remnum} \label{plusminus} Recall from the definition that a Dodgson of a graph $G$ is well-defined up to overall sign, the result of the arbitrary orientation on the edges, the ordering of the edges and vertices, and the choice of row and column deleted in constructing matrix $M_G$. Hence, it is key to fix such a matrix in calculating the five-invariant. Specifically, the choice of matrix will affect each Dodgson in a predictable manner, and calculating the same five-invariant using a different matrix will only affect the overall sign. From Lemma \ref{reindex}, the overall sign of the five-invariant may further change depending on the order of the edges in $S$. \end{remnum}

\begin{definition} Let $G$ be a graph and fix a 5-configuration $S$ in $G$. We say that $S$ \emph{splits} if, for at least one of the Dodgson polynomials $D \in D_S$, $D = 0$. If all Dodgson polynomials for a 5-configuration are non-zero, we say that $S$ is a \emph{non-splitting} 5-configuration. If $S$ splits for every possible 5-configuration $S \subseteq E(G)$, we say that $G$ itself \emph{splits}. If there exists at least one 5-configuration $S$ such that $S$ does not split, we say that $G$ is a \emph{non-splitting} graph. \end{definition}

If a fixed 5-configuration $S = \{e_1,e_2,e_3,e_4,e_5\}$ splits, then, there is a Dodgson that is equal to zero. From Lemma \ref{reindex}, it is possible to permute the indices such that $^5\Psi(e_1,...,e_5) = \pm \Psi^{e_1e_2,e_3e_4}_{e_5} \Psi^{e_1e_3e_5,e_2e_4e_5}$. As each Dodgson is, by construction, linear in each Schwinger coordinate, this five-invariant can be factored into a product of polynomials that is linear in each Schwinger coordinate.

\begin{definition} Let $G$ be a graph and $k_G$ the number of connected components in $G$. The \emph{loop number} (also known as the \emph{first Betti number} in topologically inspired literature), $h_G$, is $h_G = |E(G)| - |V(G)| + k_G$. \end{definition}

\begin{definition} Let $G$ be a graph. We say that $G$ is \emph{primitive divergent} if $|E(G)| = 2h_G$ and for any proper subgraph $g$ of $G$ with at least one edge in each connected component, $|E(g)| > 2h_g$. \end{definition}

Primitive divergent graphs are commonly studied in quantum field theory. They are graphs for which the Feynman integral diverges, but all proper subgraphs meeting the inequality in the definition have convergent Feynman integral.

The following is Corollary 52 in \cite{monster}.

\begin{prop} \label{52} If $G$ is a primitive divergent graph, then $I_G$ converges. \end{prop}

Let $G$ be a graph and apply an ordering to edges in $G$. Let $n =|E(G)|$. Regardless of our ordering of the edges of $G$ (see \cite{monster}, page 19 for an explanation), $$I_G = \int\limits_0^\infty \cdots \int\limits_0^\infty \frac{\prod_{i=1}^{n-1} \mathrm{d}\alpha_i}{\Psi^2_G} \big|_{\alpha_n = 1}$$ is equal.

The following are Corollary 129 and Proposition 130 in \cite{monster}. Recall that $I_i$ is the $i^\text{th}$ partial Feynman integral.

\begin{prop} \label{129} Let $G$ be a primitive divergent graph, $|E(G)| = n$. Then, the integral at the fifth stage of integration is $$I_5 = \int\limits_0^\infty \cdots \int\limits_0^\infty \frac{F}{^5\Psi(e_1,e_2,e_3,e_4,e_5)} \prod_{i=6}^{n-1} d\alpha_i \mid_{\alpha_n = 1}, $$ where $F$ is linear combination of trilogarithms, products of dilogarithms and logarithms, and third powers of logarithms, each in terms of the remaining Schwinger coordinates. \end{prop}

\begin{prop} \label{130} Suppose that $P_n$ is the denominator of the partial Feynman integral at the $n^\text{th}$ stage of integration. Suppose that $P_n$ factorizes into a product of linear factors in $\alpha_{n+1}$. If it is non-zero, then the denominator at the $(n+1)^\text{st}$ stage is $$P_{n+1} = \sqrt{D_{\alpha_{n+1}} (P_n)} ,$$ where $D_{\alpha_{n+1}}$ is the discriminant with respect to Schwinger coordinate $\alpha_{n+1}$. \end{prop}

This method of calculating the denominators, starting with the 5-invariant, is called \emph{denominator reduction}, introduced by Francis Brown (see \cite{monster}). If there is an ordering of the edges such that we may complete this string of calculations, the graph is \emph{denominator reducible}. Specifically, Proposition \ref{130} relies on denominator $P_n$ factoring into terms linear in $\alpha_{n+1}$.

Our understanding of the entire integral in these cases can therefore be derived from the five-invariant in the denominator of the fifth partial Feynman integral. Specifically, knowing the numerator and denominator allows us to calculate the full fraction in subsequent partial Feynman integrals.

For an arbitrary ordering of the edges of a primitive divergent graph, the first four partial Feynman integrals are given in Section 10.3 of \cite{monster}, and the fifth is found in Lemma 128. The denominators in each of these partial Feynman integrals are products of Dodgson polynomials. This, then, is the value of splitting Dodgsons in Feynman integral calculations; if a graph splits, then for any arbitrary ordering of the edges, we may immediately determine the denominators up to the sixth partial Feynman integral. As was stated before, the residues carry a great deal of information regarding the content of the Feynman integral.

Our method of approach, as stated prior, will be primarily graph theoretic. Specifically, we are looking for graphs that prove obstructions to naive edge orderings producing easily calculable sixth partial Feynman integral denominators. As such, we now introduce concepts relating to minor-closed graphs properties.

Let $G$ and $H$ be graphs and $\mathcal{F}$ a set of graphs. If a graph $G_m$ can be produced through a (possibly empty) series of  deletions and contractions of edges in $G$, we say that $G_m$ is a \emph{minor} of $G$, denoted $G_m \preceq G$. We say that $G$ is \emph{$H$-free} if $G$ does not have a minor isomorphic to $H$. We say that $\mathcal{F}$ is \emph{minor-closed} if for all $G \in \mathcal{F}$, $G_m \preceq G$ implies that $G_m \in \mathcal{F}$.

\begin{prop} For a graph $H$, the set of all $H$-free graphs is minor-closed. \end{prop}

\begin{proof} Trivially, for any graph $G$ and minor $G_m$ of $G$, $H \preceq G_m \preceq G$. The proof follows immediately. \end{proof}

The famous Robertson-Seymour theorem (\cite{robertsonseymour}) proves that for any minor-closed family of graphs, there exists a finite set of forbidden minors.

\begin{RS} For any minor-closed family of graphs $\mathcal{G}$, there exists a finite set of minors $\mathcal{F}$ such that $G$ is $F$-free for all $F \in \mathcal{F}$ if and only if $G \in \mathcal{G}$. \end{RS}

Common examples of minor-closed families include forests and planar graphs. Wagner proved that for planar graphs, the set of forbidden minors is $K_5$ and $K_{3,3}$. For forests, the only forbidden minor is trivially a single vertex with an incident loop.

\begin{definition} Let $G$ be a graph that is non-splitting for a particular 5-configuration $S$. If the 5-configuration $S$ splits in the graphs $G \cut e$ and $G \contract e$ for all edges $e \in E(G) - S$, we say that $G$ is \emph{minor-minimal non-splitting with respect to $S$}. If $e \in E(G)$ and graphs $G \cut e$ and $G \contract e$ split for all edges $e$, then we say that $G$ is \emph{minor-minimal non-splitting}. \end{definition}

To show that a specific graph is non-splitting, it is sufficient to show that there exists a 5-configuration for which the five-invariant cannot be factored to linear terms for each Schwinger coordinate. To show that a graph splits on the other hand, a Dodgson equal to zero would need to be shown for all possible 5-configurations. Further, to show that a non-splitting graph is minor-minimal, each non-isomorphic minor would need to be shown to be splitting. In any case that a graph is claimed to be splitting or minor-minimal non-splitting, verification was done using computers for large computations, and in the case of minor-minimal non-splittingness, every minor created by deleting or contracting a single edge was similarly checked. These checks are long and not enlightening, and as such are omitted here.

The following theorem is included here for motivational completeness, though the proof follows from Corollary \ref{cutablecontractable} in Chapter \ref{basics}.

\begin{theorem} \label{minorclosedproperty} Splitting is a minor-closed property. \end{theorem}

As such, there exists a set of forbidden minors for graphs that split. It is therefore natural to ask what the complete set of forbidden minors for this class of graphs is, and given the Feynman integral motivation, specifically what the simple 3-connected forbidden minors are.

%% file: files/basics.tex
%
%

\chapter{Dodgson Preliminaries}
\label{basics}

This chapter introduces some standard theorems regarding Dodgson polynomials. The results will be useful in subsequent chapters.

As stated previously, the graphs $K_{3,3}$ and $K_5$, shown in Figure \ref{k5k33}, are the two forbidden minors for planar graphs. With $K_{3,3}$ labelled as in this figure, we may calculate the five-invariants using Dodgsons. Specifically, we fix a matrix $M_G$ resulting from a particular edge ordering, orientation, and choice of deleted vertex. As stated in Remark \ref{plusminus}, using the matrix $M_G$ in each determinant calculation for the Dodgsons fixes the sign on each Dodgson, which is key in the construction of the five-invariant. Using one such matrix $M_G$, then, we calculate Dodgsons, $$ \Psi^{e_1e_2,e_4e_6}_{e_8} =  \alpha_{e_7}\alpha_{e_9}, $$ $$ \Psi^{e_1e_4e_8,e_2e_6e_8} = - \alpha_{e_3}, $$ $$ \Psi^{e_1e_4,e_2e_6}_{e_8} = - \alpha_{e_3}\alpha_{e_5} - \alpha_{e_5}\alpha_{e_9}, $$ and $$ \Psi^{e_1e_2e_8,e_4e_6e_8} = \alpha_{e_7}\alpha_{e_9}. $$ Thus, the five-invariant from this particular 5-configuration is, $$^5\Psi_{K_{3,3}}(e_1,e_2,e_4,e_6,e_8) = \pm (\alpha_{e_5}\alpha_{e_9}^2 + \alpha_{e_3}\alpha_{e_5}\alpha_{e_9} + \alpha_{e_5}\alpha_{e_7}\alpha_{e_9} + \alpha_{e_3}\alpha_{e_5}\alpha_{e_7} - \alpha_{e_3}\alpha_{e_7}\alpha_{e_9}).$$ Omitting individual Dodgson calculations, the five-invariant for $K_5$ is \begin{align*} ^5\Psi_{K_5}(e_1,e_3,e_4,e_5,e_8) = \pm \alpha_{e_2}\alpha_{e_6}\alpha_{e_7}\alpha_{e_{10}} &(\alpha_{e_6}\alpha_{e_9}^2 + \alpha_{e_2}\alpha_{e_6}\alpha_{e_9} + \alpha_{e_2}\alpha_{e_9}\alpha_{e_{10}} \\ &+ \alpha_{e_2}\alpha_{e_7}\alpha_{e_9} + \alpha_{e_2}\alpha_{e_7}\alpha_{e_{10}}) . \end{align*}

\begin{figure}[h]
  \centering
    \includegraphics[scale=1.0]{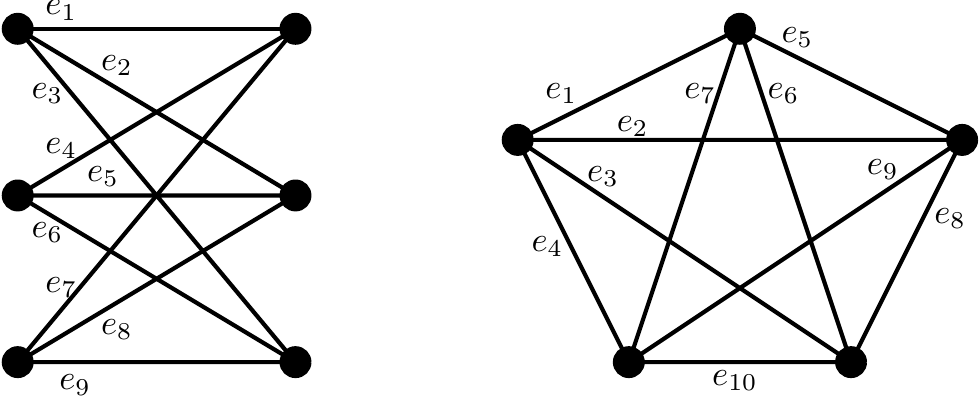}
  \caption{$K_{3,3}$ and $K_5$}
\label{k5k33}
\end{figure}

Both of these five-invariants are fully factored as written. As both polynomials would be linear in each Schwinger coordinate if these 5-configurations split, it follows that the selected 5-configurations do not split. The calculation of these five-invariants to show that $K_5$ and $K_{3,3}$ are non-splitting first appeared in \cite{monster}.

It is important to note that these are not the only non-splitting 5-configurations for $K_5$ and $K_{3,3}$. In $K_{3,3}$, edge set $\{ e_1, e_2, e_4, e_5, e_9\}$ is also a non-splitting 5-configuration. Similarly in $K_5$, 5-configuration $\{ e_1, e_4, e_5, e_8, e_{10} \}$ does not split. Up to symmetry, these are the only non-splitting 5-configurations in these graphs.

Any other minor-minimal non-splitting graph must be planar, as $K_5$ and $K_{3,3}$ are minor minimal non-splitting. By restricting our attention to planar graphs, we may introduce the graph theoretic dual as a useful tool for proving various properties of planar non-splitting graphs. For a planar graph $G$, a \emph{planar dual} of $G$ is the graph $G'$ created by fixing a planar embedding of $G$, associating a vertex to each region (including the infinite region outside of the graph), and connecting vertices in $G'$ with an edge if the associated regions in $G$ share an edge. It is a standard result that 3-connected graphs have unique planar dual (see Theorem 2.6.7 in \cite{graphbojan}). Thus, for the purposes of our main result - the list of forbidden minors for splitting, 3-connected, simple graphs - duals are unique up to isomorphism.

Since our main concern is graph minors, it is important to note that edge deletions and contractions are dual operations. Specifically, let $G$ be a planar graph and $e \in E(G)$. Let $G'$ be a dual of $G$ such that $e' \in E(G')$ is the edge corresponding to $e$. Then, graph $G' \cut e'$ is dual to $G \contract e$ and similarly $G' \contract e'$ is dual to $G \cut e$. Figure \ref{dualfig} demonstrates these dual operations, the grey blob in this diagram representing the unshown parts of graph $G$. In the special case that edge $e$ is a bridge and hence $e'$ is a loop in $G'$, we identify the vertices in $G'$ shown in this figure, and the statement still holds.

\begin{figure}[h]
  \centering
    \includegraphics[scale=1.0]{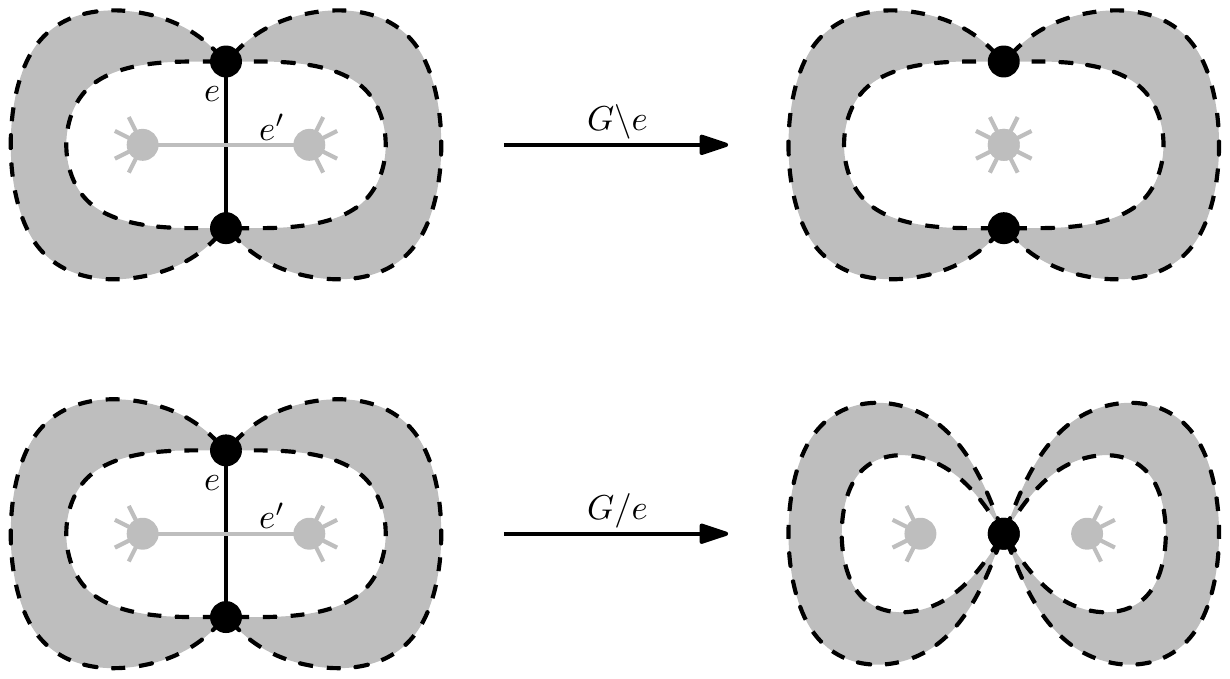}
  \caption{The delete and contract operations across dual graphs.}
\label{dualfig}
\end{figure}

\begin{prop} \label{dual} Let $G$ be a graph and $G'$ its planar dual. If the graph $G$ has a non-splitting 5-configuration, then $G'$ also has a non-splitting 5-configuration. Further, if $G$ is minor-minimal non-splitting, then $G'$ is also. \end{prop}

\begin{proof} For each edge $e \in E(G)$, let $e' \in E(G')$ be the corresponding edge of the dual. For any graph $G$, let $F(G)$ be the set of faces of $G$ in a particular planar embedding. Then, $|E(G)| = |E(G')|$, $|F(G)| = |V(G')|$, and $|V(G)| = |F(G')|$. By Euler's Formula, $|V(G)| + |F(G)| = |E(G)| + 2$ for any connected planar graph $G$. For any tree $T$ in $G$, construct a tree $T'$ in $G'$ by setting $e' \in T'$ if and only if $e \notin T$. To show that $T'$ is connected, note that we may travel between faces of $G$ by moving over edges in $E(G)$ that are not in $T$. There must exist such a path from each face to the outer face, as $T$ does not contain a cycle. By construction, then, we may travel between any two vertice $u,v \in V(G')$ along edges in $E(T')$ by first travelling to the vertex in $V(G')$ corresponding to the outer face in $G$. Thus, $T'$ is connected. Further, since $G$ has $|V(G)|$ vertices, $T$ has $|V(G)| -1$ edges. Thus, there are $|E(G)| - (|V(G)| -1)$ edges not in $T$, and $T'$ has \begin{align*}|E(G')| - (|F(G')| -1) &= (|V(G')| -2) +1 \\ &= |V(G')| -1 \end{align*} edges in $T'$. Hence, $T'$ is a tree in $G'$.

Suppose $S \subseteq E(G)$ is a non-splitting 5-configuration. Let $S' \subseteq E(G')$ such that $e' \in S'$ if and only if $e \in S$. We will show that $S'$ is non-splitting in $G'$. In calculating Dodgsons, the minors of $\Psi^{e_1e_2,e_3e_4}_{e_5}$ are dual to the minors of $\Psi^{e'_1e'_2e'_5,e'_3e'_4e'_5}$. Similarly, the minors of $\Psi^{e_1e_2e_5,e_3e_4e_5}$ are dual to the minors of $\Psi^{e'_1e'_2,e'_3e'_4}_{e'_5}$. Fix a Dodgson $D \in D_S$ and tree $T \in T_D$. Let $D'$ be the Dodgson in $D_{S'}$ such that the minors of this Dodgson are dual to the minors of $D$, as noted prior. Construct a tree $T'$ in $T_{D'}$ by placing $e' \in T'$ if and only if $e \notin T \cup S$. Then, by the above observations, $T'$ is a spanning tree of the Dodgson minors dual to the minors of $D$. It follows that if $S$ is non-splitting then $S'$ is non-splitting. 

Minor minimality of $G'$ follows immediately from the fact that $G' \cut e'$ is dual to $G \contract e$ and $G' \contract e'$ is dual to $G \cut e$. \end{proof}

\begin{prop} \label{cycle} Consider a planar graph $G$ and 5-configuration $S$. If there is an edge cut set of size one, two, or three in $G$ such that all edges of this cut are in $S$, then $S$ splits. Dually, if a subset of $S$ induces a cycle with one, two, or three edges, then $S$ splits. \end{prop}

\begin{proof} Let $S = \{e_1,e_2,e_3,e_4,e_5\}$. Suppose there is an edge cut set of size one (resp., two, three), and that it is specifically edge $e_1$ (resp., $e_1$ and $e_2$; $e_1$, $e_2$, and $e_3$). Take $D \in D_S$ such that the minors associated to the calculation of $D$ are $\mone^{\cut e_1e_2e_3}$ and $\mtwo^{\cut e_3}_{\contract e_1e_2}$. Namely, $D= \Psi^{e_1e_2e_3,e_3e_4e_5}$. No matter the number of edges in the cut set, the graph minor $\mone$ is not connected, so there of no spanning trees of $\mone$. Hence, $T_D = \emptyset$ and $D = 0$. 

Using planar duals, it follows from the previous paragraph along with Proposition \ref{dual} that if the 5-configuration $S$ induces a cycle with one, two, or three edges, then $S$ splits. \end{proof}

It is possible to prove this theorem in a stronger form; specifically without the restriction that the graph must be planar. Since we already know that $K_5$ and $K_{3,3}$ are forbidden minors, though, we may outright dismiss non-planar graphs as non-splitting. The theorem as stated suffices for our purposes.

\begin{theorem} \label{cutablecontractablegeneral} Let $G$ be a graph and $e \in E(G)$. Then $T_{G\cut e} \neq \emptyset$ if and only if there is a $T \in T_G$ such that $e \notin E(T)$. Similarly, $T_{G \contract e} \neq \emptyset$ if and only if there is a $T \in T_G$ such that $e \in E(T)$. \end{theorem}

\begin{proof} Suppose there is a $T \in T_G$ such that $e \notin E(T)$. The set of edges in $T$ forms a spanning tree of $G \cut e$. In the other direction, if $T_{G \cut e} \neq \emptyset$, then there is a spanning tree $T \in T_{G \cut e}$, and thus $T$ is a spanning tree of $G$ such that $e \notin E(T)$.

Now, suppose there is a $T \in T_G$ such that $e \in E(T)$. Let $T'$ be the edge set formed by $E(T) - \{e\}$. Then, as $T$ spanned $G$, $T'$ spans and is connected in $G \contract e$. Further, $|E(T)| = |V(G)| -1$, so $|E(T')| = |V(G \contract e)| -1$. Hence, $T'$ is a spanning tree of $G \contract e$. In the other direction, let $T' \in T_{G \contract e}$. Let $T$ be a subgraph of $G$ induced by the edge set $E(T') \cup \{e\}$. As $T'$ was a spanning tree, $T$ must be a connected spanning subgraph of $G$. Further, $|V(G\contract e)| = |V(G)| -1$ and $|E(T')| = |E(T)| -1$, so $T$ must be a spanning tree of $G$. \end{proof}

\begin{corollary} \label{cutablecontractable} Consider a graph $G$ and a non-splitting 5-configuration $S \subseteq E(G)$. Fix an edge $e \in E(G)$ such that $e \notin S$. The 5-configuration $S$ is still non-splitting in $G \cut e$ if for every Dodgson $D \in D_S$ there is a $T \in T_D$ such that $e \notin T$. Similarly, $S$ is still non-splitting in $G \contract e$ if for every Dodgson $D \in D_S$ there is a $T \in T_D$ such that $e \in T$. \end{corollary}

\begin{proof} The proof of this Corollary follows immediately from Theorem \ref{cutablecontractablegeneral}, taking trees in $T_D$ for a Dodgson $D$ and applying the theorem to the common spanning trees of minors $\mone$ and $\mtwo$. \end{proof}

As a consequence of this, if a graph $G$ splits, then for every edge $e \in E(G)$, the graphs $G \cut e$ and $G \contract e$ both split. Thus, Theorem \ref{minorclosedproperty} follows immediately; splitting is a minor closed property.

Corollary \ref{cutablecontractable} will be useful in constructing arguments around minor-minimality. Specifically, if it can be shown for a non-splitting 5-configuration $S$ that every Dodgson $D \in D_S$ has a tree $T \in T_D$ that includes (resp., does not include) an edge $e \in E(G) - S$, then $S$ is still non-splitting in $G \contract e$ (resp., $G \cut e$).

\begin{corollary} \label{loopfree} A minor-minimal non-splitting graph $G$ must be loop-free. \end{corollary}

\begin{proof} By Proposition \ref{cycle} a loop may not appear in a non-splitting 5-configuration. Trivially, a loop will never appear in a tree for any Dodgson. By Corollary \ref{cutablecontractable}, if a loop $e$ is present in a non-splitting graph it can always be deleted and $G \cut e$ will still be non-splitting. \end{proof}

\begin{prop} \label{doubleedge} Let $G$ be a planar minor-minimal non-splitting graph. For distinct vertices $u,v \in V(G)$, parallel edges $e_1$ and $e_2$ incident with vertices $u$ and $v$ may only occur if every non-splitting 5-configuration contains precisely one of the edges $\{e_1,e_2\}$. \end{prop}

\begin{proof} Suppose that there is a non-splitting 5-configuration $S$ that contains neither edge $e_1$ nor $e_2$. For any Dodgson $D \in D_S$ and any tree $T \in T_D$,  $T$ may trivially contain at most one of these edges. If $T$ contains precisely one, the choice of which edge is arbitrary, and as such if $e_1$ is in $T$, there is a tree $T' \in T_D$ such that $T' = (T -\{e_1\}) \cup \{e_2\}$. By Corollary \ref{cutablecontractable} we may delete $e_1$, contradicting minimality.

Both $e_1$ and $e_2$ may not appear in the 5-configuration, as the 5-configuration would induce a cycle with two edges and by Proposition \ref{cycle} this 5-configuration would split. \end{proof}

\begin{corollary} \label{degreetwo} In a planar minor-minimal graph with a non-splitting 5-configuration $S$, a vertex of degree two must be incident with precisely one edge in the 5-configuration. Further, suppose the vertex of degree two is incident with edges $e_1$ and $e_2$. If $e_1 \in S$, then the 5-configuration $(S \cup \{e_2\}) - \{e_1\}$ is also non-splitting. \end{corollary}

\begin{proof} This follows immediately, using planar duals. \end{proof}

As with Proposition \ref{cycle}, it is possible to prove the previous theorem and corollary without restricting to planar graphs. Again, though, this is not needed for our purposes.

\begin{prop} For a minor-minimal non-splitting graph $G$ there can be at most two parallel edges incident with the same vertex pair. \end{prop}

\begin{proof} Suppose there are at least three parallel edges $e_1,e_2,e_3 \in E(G)$. By Proposition \ref{cycle}, at most one of these edges may appear in any non-splitting 5-configuration. If precisely one of these edges is in a 5-configuration, then as in the proof of Proposition \ref{doubleedge}, one of the parallel edges not in the 5-configuration could be deleted. If none of these edges are in the 5-configuration, it also follows that all but one of these edges could be deleted. In any case, this contradicts minor-minimality. \end{proof}

A result of this theorem is that, when looking at minor-minimal non-splitting graphs without the restriction that the graph be simple, we need only consider at most double edges. Further, there may be at most five double edges in a minor-minimal non-splitting graph. By Corollary \ref{loopfree}, the minor-minimal non-splitting graphs must be loop-free. We may as such greatly restrict the appearance of multigraphs when considering minor-minimal non-splitting graphs.

%% file: files/delta-y.tex
%
%

\chapter{Non-Splitting Families}
\label{delta-y}

In this chapter, we examine the $\Delta$-Y transformation and the resulting $\Delta$-Y families of graphs. Under circumstances to be discussed, this transformation preserves non-splittingness. We also introduce the graphs $O$, $H$, and $C$, seen in Figure \ref{figbigfamily}. These are 3-connected minor-minimal non-splitting graphs that differ from each other by $\Delta$-Y transformations.

\begin{definition} A \emph{$\Delta$-Y transformation} on a graph $G$ is an exchange of an induced $K_3$ for a vertex incident with the vertices that formed the $K_3$, or a vertex of degree three for a $K_3$ contained in the neighbours of this vertex. \end{definition}

\begin{figure}[h]
  \centering
    \includegraphics[scale=1.0]{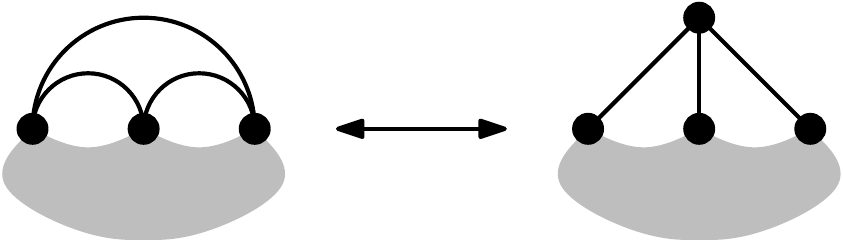}
  \caption{A $\Delta$-Y transformation.}
\label{deltay}
\end{figure}

Note that we do not specify the direction of this transformation. We will refer to either exchange as a $\Delta$-Y transformation. Cases in which specific direction is important will be addressed when needed. In any situation where vertex or edge labellings are used, we will be keeping the labellings consistent in graphs that differ by a $\Delta$-Y transformation whenever possible. That is, every vertex or edge not involved in the transformation will be given the same label in both graphs. The labels for the vertex and edges in the transformation will be given separately.

It is possible that a $\Delta$-Y  transformation will create a non-simple graph. While we have found a number of minor minimal non-splitting graphs that contain parallel edges or vertices of degree two, in Chapter \ref{trisub} we will restrict ourselves to simple, three-connected graphs.

The $\Delta$-Y transformation first appeared as a method of simplifying electrical networks. Also key to this simplification is the graph theoretic ideas of series and parallel reductions. Finding a home in graph theory, a graph $G$ is said to be $\Delta$-Y reducible if $G$ can be reduced to a single vertex using only $\Delta$-Y transformations and loop, degree-one, series, and parallel reductions. A famous result proves that planar graphs are $\Delta$-Y reducible (\cite{epifanov}). Interestingly, it was proved in \cite{truemper} that the class of $\Delta$-Y reducible graphs is minor closed. It remains an open problem to find the complete set of forbidden minors, and currently over sixty-eight billion forbidden minors have been found (\cite{yu}).

A \emph{$\Delta$-Y family} $\mathcal{F}$ is a set of multigraphs that is closed under the $\Delta$-Y operation. By construction, a $\Delta$-Y family is an equivalence class on all multigraphs, and further on multigraphs with a fixed number of edges. These families are of particular interest to us, since we do not discard parallel edges.

Our interests of course limit us to considering $\Delta$-Y families that contain minor-minimal non-splitting graphs. Figures \ref{figbigfamily} and \ref{regfamily} show collections of non-splitting graphs in two such $\Delta$-Y families. Note that both families are closed under planar duals. Not all possible $\Delta$-Y transformations are included in these figures, as not all transformations produce a graph that has a non-splitting 5-configuration or is minor-minimal non-splitting. For example, the transformation in Figure \ref{figk33minor} clearly produces a graph with a $K_{3,3}$ minor. The transformation in Figure \ref{fignotbad} produces a graph that splits for every 5-configuration. The graphs $K_5$, $K_{3,3}$, and those in Figures \ref{figbigfamily} and \ref{regfamily} are currently all known minor-minimal non-splitting graphs.

\begin{figure}[h]
  \centering
    \includegraphics[scale=1.0]{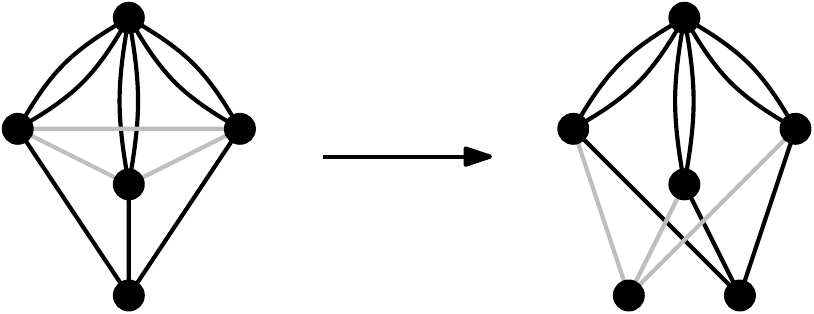}
  \caption{A $\Delta$-Y transformation from $Q$ (from Figure \ref{figbigfamily}) that produces a graph that has non-splitting 5-configurations but is not minor-minimal.}
\label{figk33minor}
\end{figure}

\begin{figure}[h]
  \centering
    \includegraphics[scale=1.0]{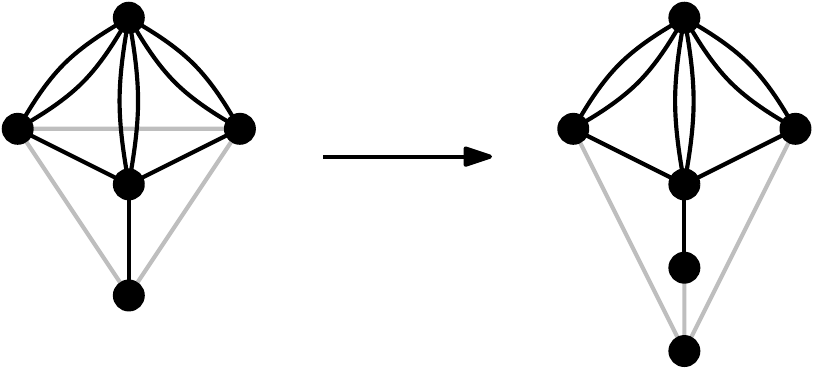}
  \caption{A $\Delta$-Y transformation from $Q$ that produces a graph that splits for all 5-configurations.}
\label{fignotbad}
\end{figure}

The full $\Delta$-Y family of the graphs in Figure \ref{figbigfamily} actually contains 15 non-isomorphic multigraphs (shown in full in Figure \ref{fullfamily}). The family containing the graphs in Figure \ref{regfamily} further contains 191 non-isomorphic multigraphs. This family as drawn has many intermediate multigraphs excluded that split or are not minor-minimal non-splitting. Both $K_5$ and $K_{3,3}$ have large $\Delta$-Y families, but are the only minor-minimal non-splitting graphs in their respective families. Specifically, the $\Delta$-Y family that includes $K_5$ has 361 non-isomorphic multigraphs. The family that includes $K_{3,3}$ has 123 non-isomorphic multigraphs.

Francis Brown found minor-minimal non-splitting graphs $K_5$, $K_{3,3}$, and $H$ in \cite{monster}. All others have been found since by Karen Yeats, Samson Black, and myself. Note that the graph $O$ is the octahedron and $C$ the cube.

\begin{theorem} \label{deltaysplits} Let $G^\Delta$ and $G^Y$ be graphs that differ by a $\Delta-Y$ transformation. Fix a 5-configuration $S$ that does not include any edges in the $\Delta-Y$, and as such is shared between the two graphs. Then $S$ splits in $G^\Delta$ if and only if it splits in $G^Y$.\end{theorem}

\begin{proof} Consider the graphs $G^\Delta$ and $G^Y$ with vertices and edges labelled as in Figure \ref{mindeltay1}. We will prove this theorem using the contrapositive, showing that $S$ is a non-splitting 5-configuration in $G^\Delta$ if and only if it is a non-splitting 5-configuration in $G^Y$. To do this, we will create trees in the minors associate with Dodgsons in $G^\Delta$ from trees in Dodgsons associated to $G^Y$ and similarly trees in Dodgsons associated to $G^Y$ from trees in Dodgsons associated to $G^\Delta$.

Suppose $S$ is a non-splitting 5-configuration in $G^\Delta$. Fix an arbitrary Dodgson $D^\Delta \in D_S^{G^\Delta}$ and a spanning tree $T \in T_{D^\Delta}$. Let $D^Y \in D_S^{G^Y}$ be the Dodgson corresponding to the same edge sets $I,J,K \subset E(G^\Delta)$ as $D^\Delta$. Create an edge set $T' \subseteq E(G^Y)$ as follows;

\begin{enumerate}
\item Set $T' := T \cap \left( E(G^{\Delta}) - \{e_1,e_2,e_3 \} \right)$
\item If $T$ contains no edges in the delta, then $T' := T' \cup \{f_1\}$
\item If $T$ contains precisely one edge in the delta, say $\{v_i,v_j\}$, then $T' := T' \cup \{v,v_i\} \cup \{v,v_j\}$
\item If $T$ contains any two edges of the delta, then include all of $f_1,f_2,f_3$ in $T'$
\end{enumerate}

\noindent As it would induce a cycle, $T$ may not contain all of the edges of the delta. In all cases, $T'$ is connected and spans the minors $\mone$ and $\mtwo$ of $D^Y$, and so $T' \in T_{D^Y}$. Thus, if $S$ is non-splitting in $G^\Delta$, then $S$ is non-splitting in $G^Y$.

Similarly, suppose $S$ is a non-splitting 5-configuration in $G^Y$. For any arbitrary Dodgson $D^Y \in D^Y_S$, fix a tree $T \in T_{D^Y}$. Again, let $D^\Delta$ be the Dodgson corresponding to the same edge sets in $G^\Delta$ and construct $T' \subseteq E(G^\Delta)$ from $T$ as follows;

\begin{enumerate}
\item Set $T' := T \cap \left( E(G^Y) - \{f_1,f_2,f_3\} \right) $
\item If precisely one of $f_1$, $f_2$, or $f_3$ is in $T$, then take no edges in the delta in $T'$
\item If two edges in the $Y$ are in $T$, say $\{v,v_i\}$ and $\{v,v_j\}$, then $T' := T' \cup \{v_i,v_j\}$
\item If all three edges in the $Y$ are in $T$, then $T' := T' \cup \{e_1,e_2\}$ 
\end{enumerate}

\noindent As $T$ must generate a spanning tree in the minors associated with $D$, at least one edge incident with $v$ must be in $T$. As before, the construction of $T'$ produces a spanning tree of the minors $\mone$ and $\mtwo$ of $D^\Delta \in D_S^\Delta$. Hence, if $S$ is non-splitting in $G^Y$, then $S$ is non-splitting in $G^\Delta$. \end{proof}

\begin{figure}[h]
  \centering
    \includegraphics[scale=1.0]{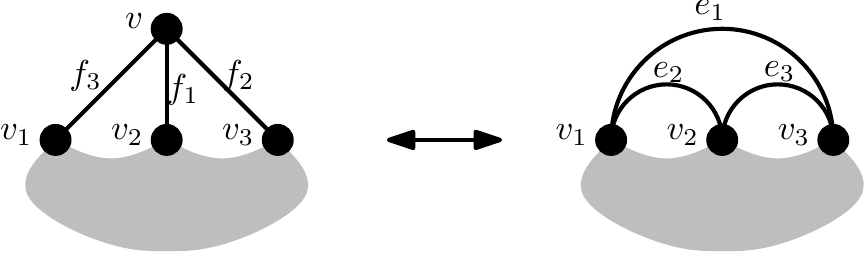}
  \caption{A graph that is minor minimal for a particular 5-configuration, but not minor-minimal non-splitting.}
\label{mindeltay1}
\end{figure}

\begin{theorem} \label{deltayminimal} Let $G^\Delta$ and $G^Y$ be two planar graphs that differ by a $\Delta$-Y transformation, and $S$ a non-splitting 5-configuration that does not include any edges in the $\Delta$-Y. Then $G^Y$ is minor-minimal non-splitting with regards to 5-configuration $S$ if and only if $G^\Delta$ is. \end{theorem}

Note that this theorem states that the graph $G$ is minor-minimal with respect to $S$. That is, for any edge $e \in E(G) - S$, the 5-configuration $S$ splits in $G \cut e$ and $G \contract e$. As an example, the graph in Figure \ref{k33minor} with a $K_{3,3}$ minor is non-splitting in the 5-configuration shown, arising from a $\Delta$-Y transformation from graph $Q$.

\begin{figure}[h]
  \centering
    \includegraphics[scale=1.0]{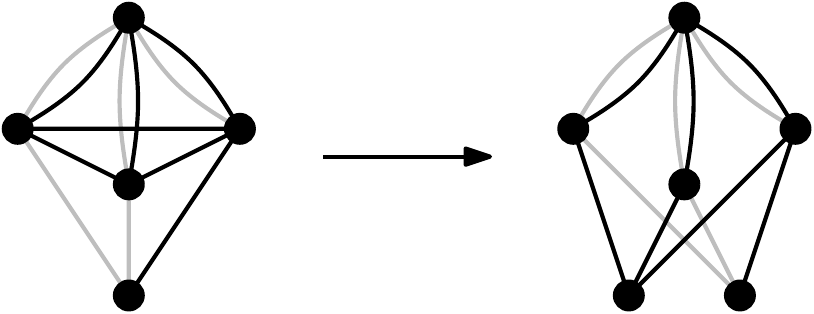}
  \caption{A 5-configuration that is minor-minimal non-splitting in a graph that is not.}
\label{k33minor}
\end{figure}

\begin{proof} Again, label the graphs as in Figure \ref{mindeltay1}. Suppose towards a contradiction that the transformation does not maintain minor-minimality. If an edge $e$ can be deleted or contracted, $e$ must be one of $\{e_1,e_2,e_3,f_1,f_2,f_3\}$, as otherwise by Theorem \ref{deltaysplits} this deletion or contraction would have been possible in the original graph. As such, we will show that any deletion or contraction of edges in the the $\Delta$ of $G^\Delta$ (resp., the Y of $G^Y$) results in a minor isomorphic to a minor of $G^Y$ (resp., $G^\Delta$), which contradicts minor-minimality of the graph with the minor-minimal non-splitting 5-configuration $S$.

First, suppose that either edge $e_i$ can be deleted or $f_j$ can be contracted, and the graph minor produced splits with respect to 5-configuration $S$. Without loss of generality, suppose $e_i = e_1$ and $f_j = f_1$. As labelled, we can see that $G^\Delta \cut e_1$ is isomorphic to $G^Y \contract f_1$. 

Next, suppose edge $e_i$ can be contracted in $G^\Delta$. Without loss of generality, suppose $e_i = e_3$. This produces parallel edges such that neither edge is in the 5-configuration. By Theorem \ref{doubleedge}, we may delete one of these edges and may hence consider $G^\Delta \contract e_3 \cut e_1$, which is isomorphic to $G^Y \contract \{f_1, f_2\}$. 

Finally, suppose we may delete an edge $f_j$ in $G^Y$, say without loss of generality $f_j = f_3$. This produces a vertex of degree two that is not incident with any edges of $S$. By Corollary \ref{degreetwo} we may contract one of these edges and if this graph is non-splitting with regards to $S$, this new minor will also be non-splitting. However, the graph $G^Y \cut f_3 \contract f_2$ is isomorphic to $G^\Delta \contract \{e_1,e_2\}$. 

In any case, then, if $G^\Delta$ (resp., $G^Y$) is not minor-minimal, then $G^Y$ (resp., $G^\Delta$) could not have been, either. Hence, $G^\Delta$ is minor-minimal non-splitting with regard to $S$ if and only if $G^Y$ is. \end{proof}

\begin{figure}[h]
  \centering
      \includegraphics[scale=1.0]{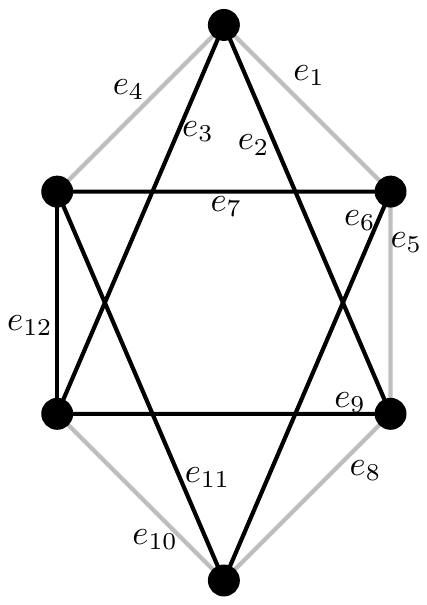}
  \caption{An edge labelling for graph $O$.}
\label{h}
\end{figure}

We may now more easily show that graphs $C$, $H$, and $O$ are non-splitting. Consider the graph $O$, labelled as in Figure \ref{h}. Fix 5-configuration $S = \{e_1,e_4,e_5,e_8,e_{10}\}$. Fully factored, the five-invariant is \begin{align*} ^5\Psi_O(e_1,e_4,e_5,e_8,e_{10}) = \pm \alpha_{e_2}\alpha_{e_6}\alpha_{e_7}\alpha_{e_9} &(\alpha_{e_3}\alpha_{e_6}\alpha_{e_9}\alpha_{e_{11}}\alpha_{e_{12}} + \alpha_{e_6}\alpha_{e_9}\alpha_{e_{11}}\alpha_{e_{12}}^2 \\ &+ \alpha_{e_3}\alpha_{e_6}\alpha_{e_{11}}\alpha_{e_{12}}^2 + \alpha_{e_2}\alpha_{e_3}\alpha_{e_6}\alpha_{e_{12}}^2 \\ &+ \alpha_{e_2}\alpha_{e_6}\alpha_{e_{11}}\alpha_{e_{12}}^2 + \alpha_{e_2}\alpha_{e_3}\alpha_{e_6}\alpha_{e_{11}}\alpha_{e_{12}} \\ &- \alpha_{e_3}^2\alpha_{e_7}\alpha_{e_{11}}\alpha_{e_{12}} + \alpha_{e_2}\alpha_{e_3}\alpha_{e_7}\alpha_{e_{12}}^2 \\ &- \alpha_{e_3}^2\alpha_{e_7}\alpha_{e_{11}}^2 + \alpha_{e_2}\alpha_{e_3}\alpha_{e_7}\alpha_{e_{11}}\alpha_{e_{12}} \\ &- \alpha_{e_3}^2\alpha_{e_9}\alpha_{e_{11}}^2 - \alpha_{e_3}\alpha_{e_9}\alpha_{e_{11}}^2\alpha_{e_{12}} \\ &- \alpha_{e_3}^2\alpha_{e_{11}}^2\alpha_{e_{12}} + \alpha_{e_2}\alpha_{e_3}\alpha_{e_{11}}\alpha_{e_{12}}^2 ). \end{align*} It follows that $O$ is non-splitting, since this five-invariant does not factor into terms linear in all Schwinger coordinates. It can further be demonstrated that $O$ is minor-minimal non-splitting. We do this by checking that all minors created by deleting or contracting a single edge in $E(O)$ split for all 5-configurations. It follows from Theorems \ref{deltaysplits} that the graphs $H$ and $C$ are non-splitting, as these graphs differ from $H$ by $\Delta$-Y transformations that include none of the edges in the 5-configuration $S$, specifically at one or both of the triangles $\{e_2,e_3,e_9\}$ and $\{e_6,e_7,e_{11}\}$. It can be shown that these are also minor-minimal non-splitting.

Recall from Proposition \ref{cycle} that if a 5-configuration contains an induced triangle or all edges incident with a vertex of degree three, then that 5-configuration splits. As such, there are ${12 \choose 5}$ possible 5-configurations in the edges of graph $O$, of which 276 will contain an induced triangle and hence will split. Interestingly, the remaining 516 5-configurations are all non-splitting. Dually, there are 516 possible 5-configurations in $C$ that do not contain all edges incident with a vertex of degree three, all of which do not split. The graph $Q$, drawn in Figure \ref{figbigfamily}, is of interest as well, as up to isomorphism there is only one non-splitting 5-configuration.

The graphs $H$, $O$, $C$, $K_5$, and $K_{3,3}$ have all been shown to be non-splitting. The remaining graphs that we claim are minor-minimal non-splitting all contain either a two vertex cut or double edge. As stated prior, our main result is to demonstrate that all non-splitting 3-connected simple graphs contain at least one of $H$, $O$, $C$, $K_5$, or $K_{3,3}$ as a minor. For the sake of brevity, then, we will omit five-invariant calculations that prove that the graphs with two vertex cuts or parallel edges are non-splitting.

Characterizations of minor-free graphs are common in graph theory. Similar to finding the set of forbidden minors for a minor-closed property of graphs, this involves describing all graphs that are free of a particular minor. As such, there are a number of results regarding graphs that are planar, $H$-, $O$-, or $C$-free.

There are characterizations in \cite{octahedron} and \cite{cube} of 4-connected graphs that are octahedron-free and cube-free, respectively. The results in \cite{octahedron} in particular demonstrate that we need only consider graphs that are at most 3-connected, which will be discussed in greater detail in Chapter \ref{cuts}. There is no approach for graphs with two or three vertex cuts in \cite{octahedron}, though. There is a characterization of cube-free 3-connected graphs in \cite{cube}, but it inductively relies on large checks for tree-decompositions.

In \cite{octfree}, the author demonstrates that a graph is $O$-free if and only if it it can constructed from a particular set of graphs using $k$-sums, $0 \leq k \leq 3$. Unfortunately, this construction method allows for deleting individual edges. As such, this was not useful for our approach. Specifically, deleting an edge from a non-splitting graph may produce a graph that splits.

The authors of \cite{cube-free} produce a characterization of 3-connected planar $C$-free graphs using minimal non-trivial three vertex cuts. To do this, they produce a small set of graphs such that this three vertex cut must be a spanning subgraph of one of the graphs in this set. Our results in Chapter \ref{trisub} are partially based on these results. In particular, the graph $S_2$, introduced in Figure \ref{s2new}, is a spanning subgraph of one of their graphs, shown in Figure 3 of \cite{cube-free}. One may use their results to construct planar cube-free graphs of arbitrary size by inserting one of their spanning subgraphs into another planar cube-free graph. Unfortunately, this method allows for chains of arbitrary length and little discernible structure, and these chains may further produce graphs with $H$ minors.

\begin{figure}[h]
  \centering
      \includegraphics[scale=0.65]{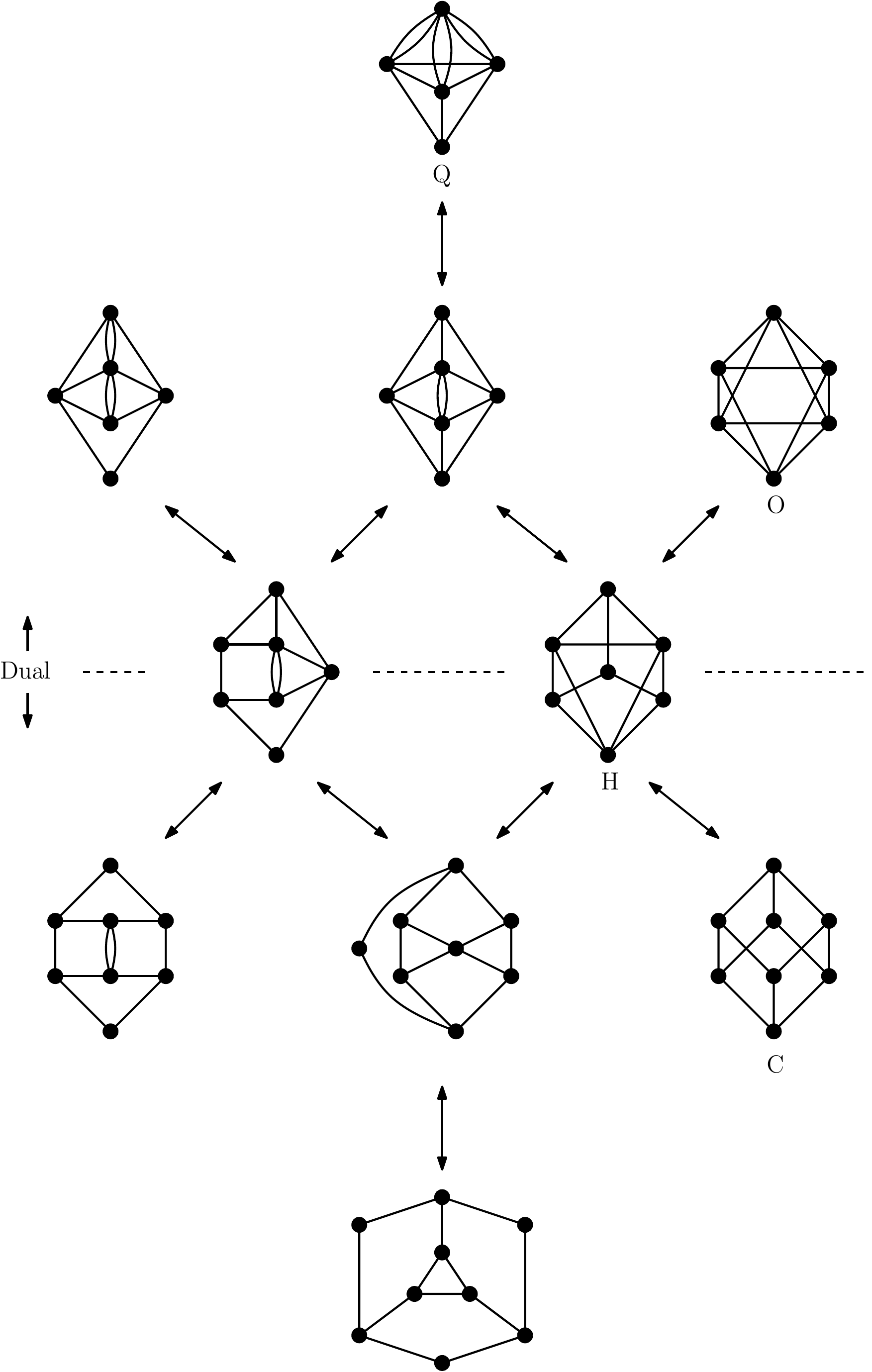}
  \caption{The minor-minimal non-splitting members of the $\Delta-Y$ family that contains $O$, $H$, and $C$.}
\label{figbigfamily}
\end{figure}

\begin{figure}[h]
  \centering
      \includegraphics[scale=0.63]{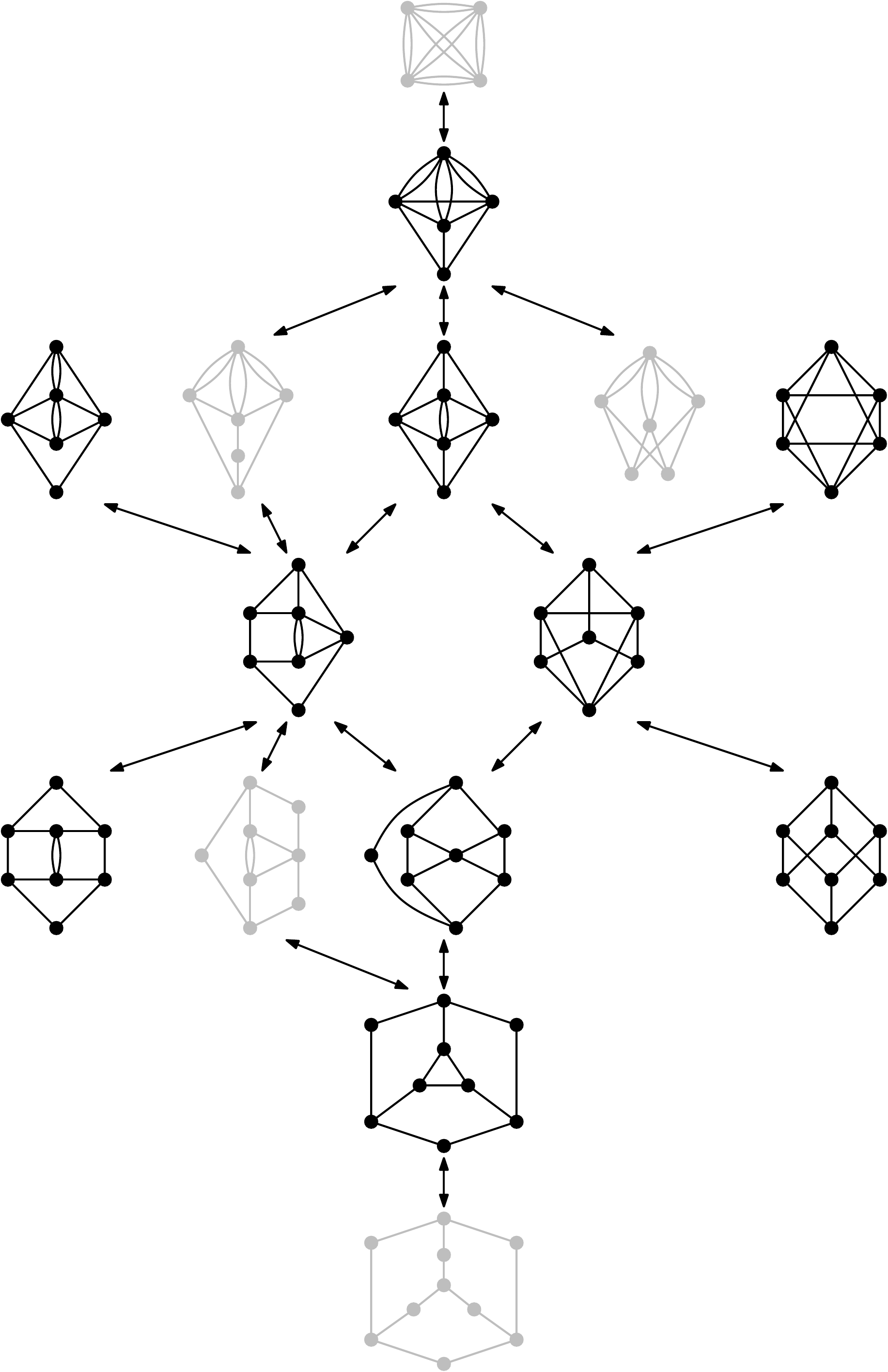}
  \caption{The full $\Delta$-Y family that contains $O$, $H$, and $C$. Graphs not in the previous figure are shaded grey. The graph previously noted as having a $K_{3,3}$ minor is non-splitting, all other new graphs split.}
\label{fullfamily}
\end{figure}

\begin{figure}[h]
  \centering
      \includegraphics[scale=1.0]{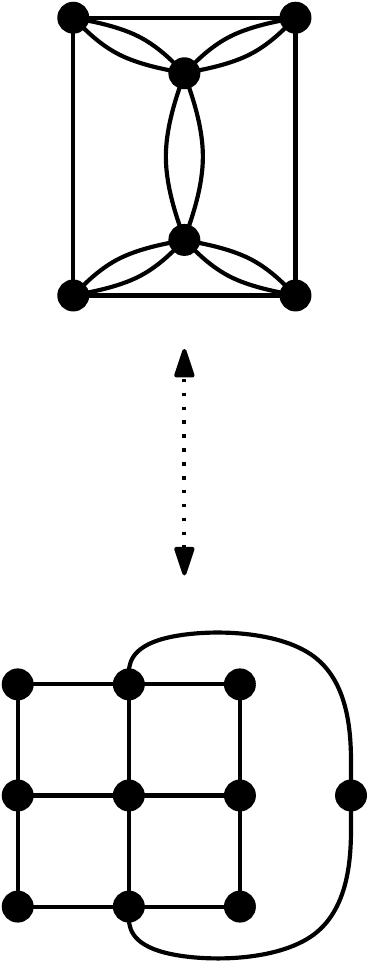}
  \caption{The members of a second $\Delta$-Y family that are known to be minor-minimal non-splitting. Note that both graphs are non-simple or have a two vertex cut, and that these graphs are dual to each other.}
\label{regfamily}
\end{figure}

%% file: files/connectivitynew.tex
%
%

\chapter{Graph Connectivity}
\label{cuts}

In this chapter, we examine how the connectivity of a graph affects splittingness. Since we are looking for spanning trees of graph minors, it seems logical that vertex cut sets would influence the existence of such trees, and especially affect how the tree appears in the components created by such a cut. This will allow for great restrictions on the connectivity of graphs that split.

Let $G$ be a graph. A \emph{connected component} $H$ of $G$ is a maximally connected subgraph. Suppose $G$ is connected. A set of vertices $S \subset V(G)$ is a \emph{vertex cut set} if the graph induced by $V(G) - S$ is disconnected or trivial. If $|S| =1$, then the vertex $v \in S$ is a \emph{cut vertex}. The \emph{vertex connectivity}, $\kappa(G)$, of a connected graph is the minimum number of vertices whose deletion disconnects graph $G$ or creates a trivial graph. For any integer $n$, if $\kappa(G) \geq n$, we say that $G$ is $n$-connected.

The following is a famous theorem by Menger, useful in proving $k$-connectivity. The corollary follows logically (see Theorem 1.4.7 in \cite{graphbojan}).

\begin{Menger} Let $G$ be a finite undirected graph and $v_1,v_2 \in V(G)$ two distinct, non-adjacent vertices. The minimum number of vertices which must be removed to disconnect $v_1$ and $v_2$ is equal to the maximum number of vertex-independent paths between $v_1$ and $v_2$. \end{Menger}

\begin{corollary} \label{mengerscor} Let $G$ be a $k$-connected graph, $v \in V(G)$, and $A \subset V(G)$ such that $|A| = k$ and $v \notin A$. There are $k$ paths in $G$, call them $P_1,...,P_k$, such that path $P_i$ is from $v$ to vertex $a_i \in A$ for $1\leq i \leq k$ and for any two distinct paths $P_j$ and $P_k$, $V(P_j) \cap V(P_k) = \{v\}$. \end{corollary}

The following definition arises from the connected components of a graph after a vertex cut.

\begin{definition} Let $G$ be a connected graph with a vertex cut set $\{ v_1,v_2,...,v_n \} \subseteq V(G)$ such that the deletion creates connected components $G_1, ... , G_m$. Define $C_i$ to be the graph induced by $V(G_i) \cup \{ v_1,v_2,...,v_n \}$ for $1 \leq i \leq m$. Call these the \emph{full components}. \end{definition}

Full components will be useful when considering 5-configurations and vertex cuts. Specifically, if $G$ is a graph and $K \subset V(G)$ is a vertex cut set, the graph induced by $V(G) -K$ will have fewer edges than $|E(G)|$. When considering the behaviour of a non-splitting 5-configuration in a graph with a vertex cut set, it will be paramount that we maintain all the edges incident with vertices in the set $K$. Unfortunately, if $v_i,v_j \in V(G) \cap K$ are two distinct vertices and $\{v_i,v_j\} \in E(G)$, then this edge will appear in every full component created by vertex cut set $K$. In the particular case where a 5-configuration contains such an edge, then, we must adjust our proof techniques accordingly.

We will now begin restricting the connectivity of graphs that are minor-minimal non-splitting.

Let $C_m$ be the cycle on ordered vertices $\{0,1,...,m-1\}$. The graph $C^2_{m}$ is created from this cycle by adding the edges $\{i, i+2\}$ for all $i \in \{0,1,...,m-1\}$ where addition is modulo $m$. This graph is called the \emph{square of the cycle}. Trivially, the graph $C^2_{2k+1}$, $k \geq 2$, is nonplanar. The following is Theorem 5.1 in \cite{octahedron}.

\begin{theorem} \label{notfour} Let $G$ be a $4$-connected graph. Then $G$ either contains a minor isomorphic to the octahedron $O$ or is isomorphic to the square of on odd cycle, $C^2_{2k+1}$ for some $k \geq 2$.\end{theorem}

A prior result which is also sufficient for our needs was proved by Halin in \cite{halin}.

\begin{theorem} If $G$ is a simple graph with minimum vertex degree four, then $G$ has a minor isomorphic to the octahedron or $K_5$. \end{theorem}

It was demonstrated in Chapter \ref{basics} that $K_5$ and $K_{3,3}$ are non-splitting, and a quick check reveals that they are minor-minimal non-splitting. In Chapter \ref{delta-y} the graph $O$ was shown to be non-splitting, and again is in fact minor-minimal non-splitting. It follows from Theorem \ref{notfour} that any further forbidden minors may be at most 3-connected. The following theorem proves that a graph with a cut vertex cannot be minor-minimal non-splitting, and we are left only needing to consider graphs $G$ with $\kappa(G)$ equal to either two or three.

\begin{theorem} A minor-minimal non-splitting graph cannot have a cut vertex. \end{theorem}

\begin{proof} Suppose towards a contradiction that graph $G$ is a minor-minimal non-splitting graph with a cut vertex, $v$. Let $G_1,...,G_n$ be the connected components of the graph induced by $V(G) - \{v\}$. Fix a non-splitting 5-configuration $S = \{e_1,e_2,e_3,e_4,e_5\}$.

\startingcases
\begin{case} A full components $C_j$ contains no edges of the 5-configuration. \end{case} 

Fix a spanning tree of $C_j$, call it $t_C$. As $C_j$ is unaffected by the deletions and contractions in taking the minors $\mone$ and $\mtwo$, for any Dodgson $D \in D_S$ and tree $T \in T_D$, we may construct a tree $T'$ such that $T'$ is induced by the edge set $(T - (E(C_j) \cap T)) \cup E(t_C)$. Whence, $T' \in T_D$, and as $D$ was arbitrary, by Corollary \ref{cutablecontractable} we may contract all the edges in $E(t_C)$ and delete all edges in $E(C_j) - E(t_C)$, contradicting minimality.

\begin{case} There is a full component $C_j$ with one or two edges in the 5-configuration. \end{case}

Without loss of generality suppose $C_j$ contains edge $e_1$ only (resp., $e_1$ and $e_2$). Let $e_1 = \{v_1,v_2 \}$. Consider $D = \Psi^{e_1e_2e_3,e_2e_4e_5} \in D_S$ and minors $\mone^{\cut e_1e_2}$, $\mtwo^{\cut e_2}_{\contract e_1}$. For any spanning tree $T \in T_D$, $T$ must contain a path from $v$ to $v_1$ and from $v$ to $v_2$ in $\mone$. This path may be trivial, and allows for $v = v_j$, $j \in \{1,2\}$. In $\mtwo$, these paths will necessarily create a cycle. As such, it is not possible to have a common spanning tree between $\mone$ and $\mtwo$, and hence $D = 0$, a contradiction. 

\exitingcases

Since there are at least two full components, any distribution of edges of $S$ in $G$ will have a full component containing at most two edges in the 5-configuration. Since $S$ is a non-splitting 5-configuration, it must be the case that $G$ is not minor-minimal, and as such no minor-minimal non-splitting graphs may have cut vertices. \end{proof}

Hence, a minor-minimal non-splitting graph must be at least 2-connected. In Chapter \ref{delta-y}, we specifically saw vertices of degree two in minor-minimal non-splitting graphs. We know then that there are in fact graphs $G$ such that $G$ is minor-minimal non-splitting and $\kappa(G) = 2$. We may, however, restrict the appearance of minor-minimal non-splitting graphs with two vertex cuts, and further provide a restriction on 5-configurations in special cases of graphs with three vertex cuts.

\begin{prop}\label{prop1} Suppose $G$ is a graph with a two vertex cut $\{v_1,v_2\}$. Suppose this cut induces full component $C_j$. Let $S = \{e_1,e_2,e_3,e_4\} \subset E(G)$ be a fixed set of four edges such that $e_1$ and $e_2$ are in $E(C_j)$, but edges $e_3$ and $e_4$ are not. The Dodgson $\Psi^{e_1e_2,e_3e_4} = 0$. \end{prop}

\begin{proof} Let $D = \Psi^{e_1e_2,e_3e_4}_{}$ and consider its minors $\mone^{\cut e_1e_2}$ and $\mtwo_{\contract e_1e_2}$. Suppose $T \in T_D$. Let $|V(C_j)|=v$. Then, in $\mtwo$ the tree $T$ may contain at most $(v-2)-1$ edges in $E(C_j)$ as the full component appears in this minor. Since $T$ is a spanning tree in both $\mone$ and $\mtwo$, there will be three fewer edges in $T \cap E(C_j)$ than vertices in $V(C_j) \cap V(\mone)$. It is impossible that all vertices in $C_j$ have paths to $v_1$ or $v_2$ in $\mone$ and as such $T$ is not connected in $\mone$, a contradiction. \end{proof}

\begin{corollary} \label{corprop1} Suppose $G$ is a graph with a two vertex cut that separates component $G_j$. Consider full component $C_j$. If $S = \{e_1,...,e_5\}$ is a 5-configuration such that $e_1,e_2 \in E(C_j)$ but $e_4,e_5 \notin E(C_j)$, then $S$ splits. \end{corollary}

\begin{proof} This follows immediately from Proposition \ref{prop1}, as Dodgson $\Psi^{e_1e_2e_3,e_3e_4e_5}$ in $G$ is equal to Dodgson $\Psi^{e_1e_2,e_4e_5}$ in $G \cut e_3$, where $\{v_1,v_2\}$ must still be a two vertex cut (though not necessarily a minimal two vertex cut). \end{proof}

\begin{prop} \label{prop2} Let $G$ be a minor-minimal non-splitting graph with a two vertex cut $\{v_1,v_2\}$ and let $C_j$ be a full component induced by this cut. Let $S$ be a non-splitting 5-configuration such that $E(C_j) \cap S = \{e\}$. Fix a path $P_1$ in $C_j$ from $v_1$ to $v_2$ such that $e \in E(P_1)$. If a similar path exists that does not contain edge $e$, let $P_2$ be a path from $v_1$ to $v_2$, $e \notin P_2$, such that the graph induced by $(E(P_1) \cup E(P_2)) - \{e\}$ does not contain a cycle. Then, $|E(P_1)| \leq 2$. If a path $P_2$ exists, then $|E(P_2) - E(P_1)| = 1$. \end{prop}

\begin{proof} First, suppose that path $P_2$ exists. Fix a spanning tree $t_C$ of $C_i$ such that $$(E(P_1) \cup E(P_2)) - \{e\} \subseteq t_C.$$ Suppose further that $|E(P_1)| \geq 2$. Let $f_1$ be an edge on path $P_1$, $e \neq f_1$. Let $f_2 \in E(P_2) - E(P_1)$. Take an arbitrary Dodgson $D \in D_S$ and tree $T \in T_D$. As $f_1$ may or may not lie on path $P_2$, both cases must be considered, as shown in Figure \ref{2cuttrees}.

\begin{figure}[h]
  \centering
      \includegraphics[scale=1.0]{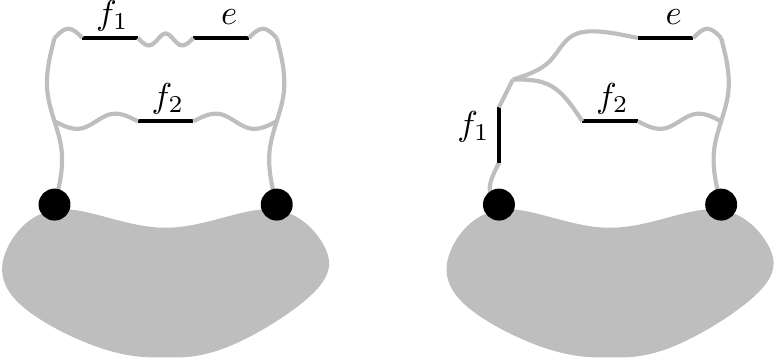}
  \caption{Possible two vertex cut full components up to symmetry, assuming path $P_2$ exists. Arbitrary paths, coloured grey, may have length zero.}
\label{2cuttrees}
\end{figure}

In either case, the edge set $T \cap E(C_j)$ either induces a connected graph or a graph with two connected components in minors $\mone$ and $\mtwo$. Since $|T \cap E(C_j)|$ is constant, the number of components in $C_j$ cannot change between minors $\mone$ and $\mtwo$ in which edge $e$ is deleted or contracted in both minors. We will show that we can create a tree in $T_D$ by replacing edge set $T \cap E(C_j)$ with a subset of $E(t_C)$.

\startingcases
\begin{case} Dodgson $D$ has minors $\mone^{\cut e}$ and $\mtwo_{\contract e}$. \end{case}

Figure \ref{2cuttrees1} shows the two minors with edge $f_2$ deleted in both cases. Let $|V(C_j)| = v$. Then there are $v$ vertices in this full component as it appears in $\mone$, and $v-1$ vertices in this full component as it appears in $\mtwo$. Hence, tree $T$ has a bound on the edges $|T \cap E(C_j)| \leq v-2$, and so $T \cap E(C_j)$ must induce a graph with precisely two connected components. In $C_j$ as it appears in the minors, there will be a path from $v_1$ to $v_2$ in $\mtwo$ and no such path in $\mone$. We may thus generate a tree in $T_D$ by replacing edges $T \cap E(C_j)$ with edges $E(t_C) - \{f_2\}$ in $T$.

\begin{figure}[h]
  \centering
      \includegraphics[scale=0.7]{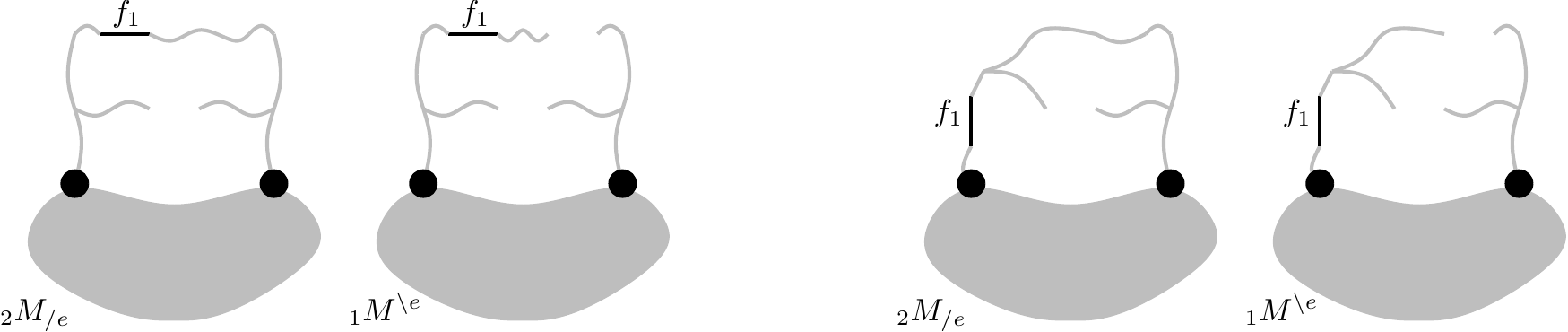}
  \caption{Minors $\mone$ and $\mtwo$ of both possible cases.}
\label{2cuttrees1}
\end{figure}

\begin{case} Dodgson $D$ has minors $\mone_{\contract e}$ and $\mtwo_{\contract e}$. \end{case}

Figure \ref{2cuttrees2} shows ways to replace $T \cap E(C_j)$ whether or not the graph induced by $T \cap E(C_j)$ is connected in minors $\mone$ and $\mtwo$. That is, if $T \cap E(C_j)$ is connected in $\mone$ and $\mtwo$, we may replace it with $E(t_C) - \{f_2\}$. If $T \cap E(C_j)$ is not connected, replace it with $E(t_C) - \{f_1,f_2\}$.

\begin{figure}[h]
  \centering
      \includegraphics[scale=0.7]{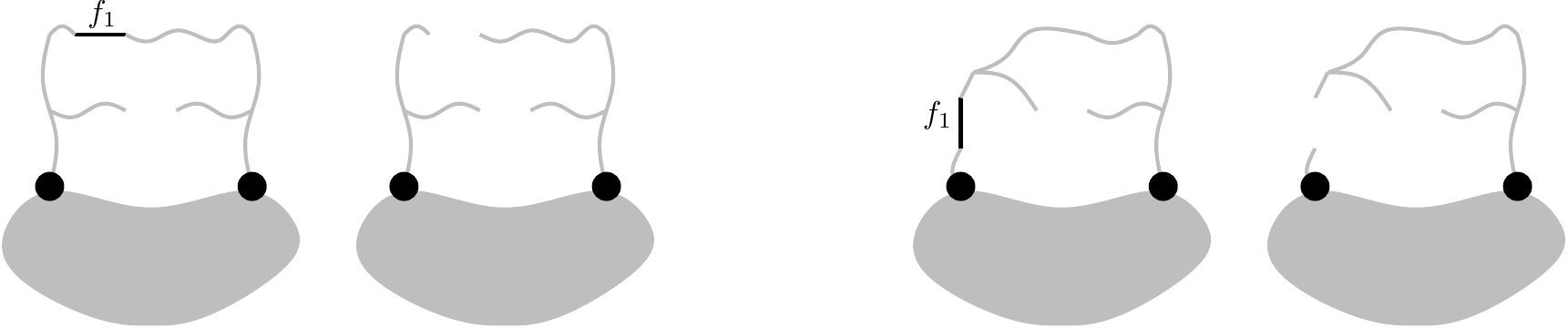}
  \caption{Edge sets that allow for this full component to be either connected or not when edge $e$ is contracted in both minors.}
\label{2cuttrees2}
\end{figure}

\begin{case} Dodgson $D$ has minors $\mone^{\cut e}$ and $\mtwo^{\cut e}$. \end{case}

These cases are shown in Figure \ref{2cuttrees3}. If $T \cap E(C_j)$ is connected in $\mone$ and $\mtwo$, we may replace it with $E(t_C)$. If $T \cap E(C_j)$ is not connected, replace it with $E(t_C) - \{f_2\}$. Note in particular that this is the only case in which we may want to include edge $f_2$ in the replacement. 

\begin{figure}[h]
  \centering
      \includegraphics[scale=0.7]{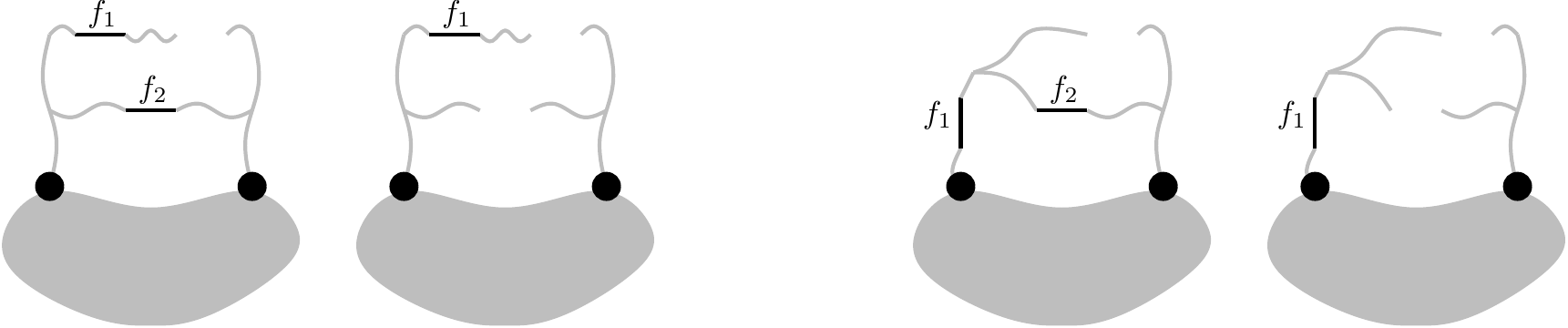}
  \caption{Edge sets that allow for this full component to be either connected or not when edge $e$ is deleted in both minors.}
\label{2cuttrees3}
\end{figure}

\exitingcases

In any case, every edge in $E(C_j) - (E(t_C) \cup \{e\})$ can be deleted by Theorem \ref{cutablecontractable} and similarly any edge in $E(t_C) - \{e,f_1,f_2\}$ can be contracted. Thus, the path $P_1$ can contain at most two edges in a minor-minimal non-splitting graph. If a path $P_2$ exists, then $|E(P_2) - E(P_1)| = 1$. 

Suppose now that there is no path $P_2$. As such, every path in $C_j$ from $v_1$ to $v_2$ must include edge $e$, and path $P_1$ must contain at least two edges. Again, we select an edge $f_1 \in E(P_1)$ such that $f_1 \neq e$. Up to symmetry, this is shown in Figure \ref{2cutnop2}.

\begin{figure}[h]
  \centering
      \includegraphics[scale=1.0]{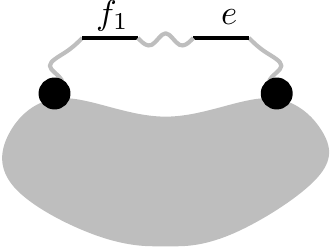}
  \caption{The possible full component, up to symmetry, if path $P_2$ does not exist. Arbitrary paths, coloured grey, may again have length zero.}
\label{2cutnop2}
\end{figure}

Construct a tree $t_C$ in $C_j$ such that $E(P_1) \subseteq E(t_C)$. Let $D \in D_S$ be arbitrary, and pick a $T \in T_D$. If up to relabelling edge $e$ is deleted in minor $\mone$ and contracted in $\mtwo$, then as in Case 1 we may replace $T \cap E(C_j)$ with $E(t_C) - \{e\}$. If $e$ is deleted in both minors, the path $P_1$ no longer connects $v_1$ and $v_2$ in $C_j$ in either minor, and we can replace $T \cap E(C_j)$ with $E(t_C) - \{e\}$ to create a tree in $T_D$. If $e$ is contracted in both minors, we can replace $T \cap E(C_j)$ with either $E(t_C) - \{e\}$ or $E(t_C) - \{e,f_1\}$, depending on whether or not $T \cap E(C_j)$ is connected in $\mone$ and $\mtwo$. Then, we may delete all edges in $E(C_j) - E(t_C)$ and contract all edges in $E(P_1) - \{e,f_1\}$. Again, $|E(P_1)| = 2$.

We now consider the case that $|E(P_1)|<2$. Suppose that edge $e = \{v_1,v_2\}$. Then, it is not possible to choose an edge $f_1$. In this case, though, a path $P_2$ must exist since this two vertex cut was nontrial and minimal. Then, the proof is similar; we potentially keep edge $f_2$ in the tree to maintain connectivity when $e$ is deleted in minors $\mone$ and $\mtwo$, and delete it in all other cases. We could again contract up all edges in $E(P_2) - \{f_2\}$, and as such $V(C_j) = \{v_1,v_2\}$, contradicting the fact that this is a connected component induced by the two vertex cut $\{v_1,v_2\}$. Whence, this situation cannot arise. \end{proof}

\begin{corollary} \label{14split} Let $G$ be a minor-minimal non-splitting graph with a two vertex cut $\{v_1,v_2\}$. Let $S$ be any non-splitting 5--configuration. Then the graph induced by this cut set has precisely two connected components, $G_1$ and $G_2$, and the full components $C_1$ and $C_2$, up to relabelling, have $|E(C_1) \cap S| = 1$ and $|E(C_2) \cap S| = 4$. Furthermore, if $\{v_1,v_2\} \in E(G)$, then it is not in $S$. \end{corollary}

\begin{proof} We will first restrict the number of edges in $S$ that may appear in any full component.

\startingcases
\begin{case} Full component $C_j$ contains no edges in $S$. \end{case}

Fix a spanning tree of $C_j$, call it $t_C$. Since this tree must contain a path from $v_1$ to $v_2$, let $e \in E(t_C)$ be an edge that lies on this path. For a Dodgson $D \in D_S$ and tree $T \in T_D$, vertices $v_1$ and $v_2$ may or may not be connected in the $C_j$-spanning graph induced by $T \cap E(C_j)$. If $v_1$ and $v_2$ are connected in the graph induced by $E(C_j) \cap T$, there is a tree $T' \in T_D$ induced by edge set $((E(G) - E(C_j)) \cap T) \cup E(t_C)$. If $v_1$ and $v_2$ are not connected in the graph induced by $E(C_j) \cap T$, there is a similar tree $T' \in T_D$ induced by edge set $((E(G) - E(C_j)) \cap T) \cup (E(t_C) - \{e\})$. As Dodgson $D$ was arbitrary, we may contract every edge in $E(t_C) - \{e\}$, contradicting minor-minimality.

\begin{case} Full component $C_j$ contains two or three edges in $S$. \end{case}

By Corollary \ref{corprop1}, there will be a Dodgson that will split given the positioning of the edges in $S$.

\exitingcases

It follows that all full components must contain either one, four, or five of the edges in $S$. If $\{v_1,v_2\} \notin E(G)$ or $\{v_1,v_2\} \in E(G)$ but $\{v_1,v_2\} \notin S$, it must be the case that there are two full components; a full component containing one of the edges in $S$ and a full component the remaining four edges in $S$. If $\{v_1,v_2\} \in S$, it could be the case that there are any number of full components containing one edge in $S$, and one full component containing five edges in $S$. Each full component containing one edge in $S$ would then have a non-trivial path from $v_1$ to $v_2$ that contains no edges in $S$ and is disjoint from a path in the full component containing edge $\{v_1,v_2\}$. By Proposition \ref{prop2}, we may contract all edges except for one on this path, contradicting minor-minimality. Hence we must have precisely two connected components, one containing one edge of $S$, the other containing the remaining four. \end{proof}

\begin{corollary} \label{trickythreecut} Let $G$ be a graph with a three vertex cut, $S \subseteq E(G)$ a 5-configuration, and $e_1 \in S$. If $G \cut e_1$ or $G \contract e_1$ has a two vertex cut that produces two connected components such that one full component contains precisely two of the remaining edges in $S$, then $S$ splits. \end{corollary}

\begin{proof} Suppose deleting edge $e_1$ produces a graph with a two vertex cut, and this two vertex cut induces a full component $C_1$ containing $e_2,e_3 \in S$ but no other edges in $S$. By Proposition \ref{prop1}, the Dodgson $\Psi^{e_2e_3,e_4e_5} = 0$ in $G \cut e_1$. As the minors are isomorphic, Dodgson $\Psi^{e_1e_2e_3,e_1e_4e_5} = 0$ in $G$. Hence $S$ splits.

If contracting edge $e_1$ produces a graph with such a two vertex cut, the proof is similar, as by Proposition \ref{prop1} the Dodgson $\Psi^{e_2e_3,e_4e_5} = 0$ in $G \contract e_1$. \end{proof}

\begin{corollary} \label{howitlooks} Let $G$ be a minor-minimal non-splitting graph and $S$ a non-splitting 5-configuration. Suppose $G$ has a two vertex cut at vertices $v_1$ and $v_2$. Then, the full component with precisely one edge in $S$ must be a proper subgraph of the full component seen in Figure \ref{fig2cutminor}. \end{corollary}

\begin{figure}[h]
  \centering
      \includegraphics[scale=1.0]{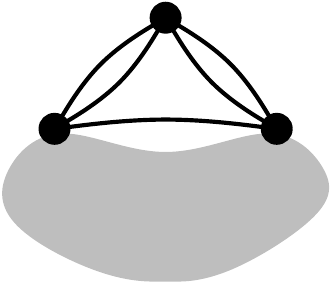}
  \caption{A full component such that any minor-minimal non-splitting graph with a two vertex cut has a proper subgraph of this as a full component.}
\label{fig2cutminor}
\end{figure}

\begin{proof} By Corollary \ref{14split}, there must be two full components created by this two vertex cut, say $C_1$ and $C_2$, such that one full component contains precisely one edge in $S$. Without loss of generality, suppose $|E(C_1) \cap S| = \{e_1\}$. By Corollary \ref{14split}, $e_1 \neq \{v_1,v_2\}$. Let $P_1$ be a path from $v_1$ to $v_2$ in $C_1$ that contains edge $e_1$, and, if it exists, let $P_2$ be a similar path that does not contain edge $e_1$. Then, from Proposition \ref{prop2}, minor-minimality of $G$ means that path $P_1$ must have precisely two edges, and path $P_2$, if it exists, must have a single edge that is not in $P_1$. In any case, this must be a subgraph of the full component in Figure \ref{fig2cutminor}, as desired. 

Specifically, if a path $P_2$ exists, full component $C_1$ must be isomorphic to one of the full components shown in Figure \ref{yesp2}. If no such path exists, then full component $C_1$ must be isomorphic to the full component in Figure \ref{nop2}. \end{proof}

\begin{figure}[h]
  \centering
      \includegraphics[scale=1.0]{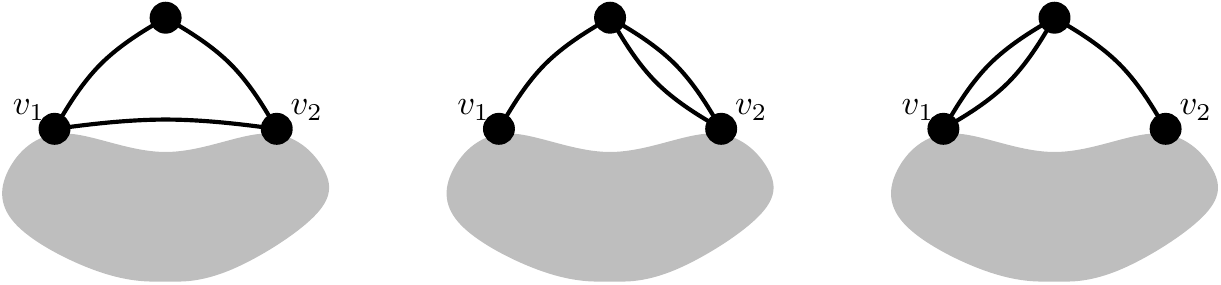}
  \caption{Possible full components induced by a two vertex cut if a path exists between vertices $v_1$ and $v_2$ that does not include edge $e_1$.}
\label{yesp2}
\end{figure}

\begin{figure}[h]
  \centering
      \includegraphics[scale=1.0]{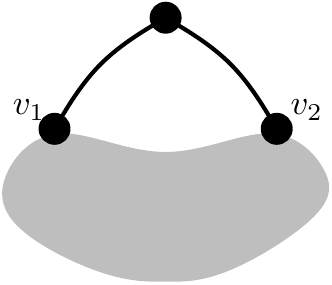}
  \caption{A full component induced by a two vertex cut if no such path exists.}
\label{nop2}
\end{figure}

%% file: files/trisubnew.tex
%
%

\chapter{Full Component Constructions}
\label{trisub}

In this chapter, all graphs are assumed to be simple and 3-connected graphs. We prove that if such a graph is non-splitting then it must have a minor isomorphic to $H$, $O$, $C$, $K_5$, or $K_{3,3}$.

\begin{prop} \label{small} Let $G$ be a planar graph with at most five vertices. Then $G$ splits. \end{prop}

\begin{proof} If $G$ has five vertices, it must be a spanning subgraph of $K_5$ minus an edge. This graph splits. Any simple graph with at most four vertices must be a minor of $K_4$, which also splits. Hence, $G$ splits. \end{proof}

It immediately follows from Proposition \ref{small} and Theorem \ref{notfour} that any non-splitting graph $G$ that is 3-connected, planar, and $O$-free must have a three vertex cut that disconnects the graph. Whence, any 3-connected minor-minimal non-splitting graphs not isomorphic to $K_5$, $K_{3,3}$, $C$, $H$, or $O$ must contain a three vertex cut.

Consider the graph in Figure \ref{s2new}, call it $S_2$. Labelled as in this figure, we are interested in graphs in which $S_2$ appears as a full component, specifically such that vertices $v_1$, $v_2$, and $v_3$ form the vertex cut set that isolates this component.

\begin{figure}[h]
  \centering
    \includegraphics[scale=1.0]{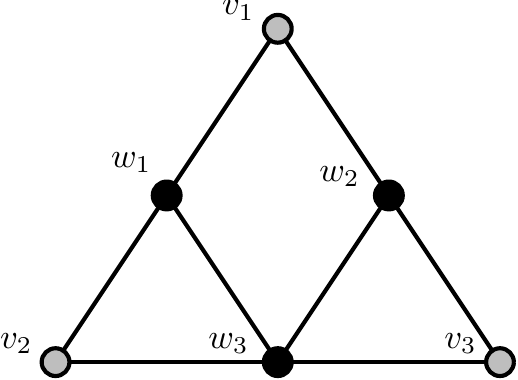}
  \caption{$S_2$}
\label{s2new}
\end{figure}

\begin{definition} Consider a graph $G$ with a three vertex cut. Let $C_i$ be a full component created by this cut.  If there is a minor of $C_i$ isomorphic to $S_2$ such that this minor connects to the rest of the graph as in Figure \ref{wellconnected}, we say that this $S_2$ minor is \emph{well connected}. \end{definition}

\begin{figure}[h]
  \centering
    \includegraphics[scale=0.9]{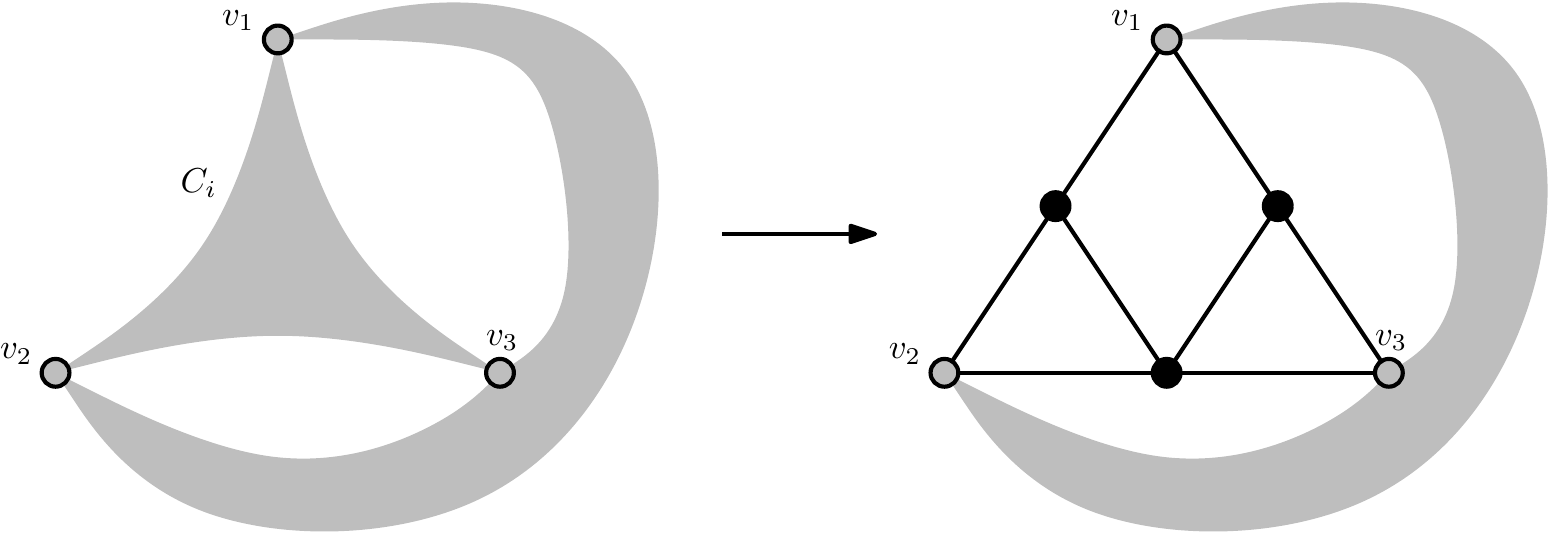}
  \caption{The appearance of a well connected $S_2$ minor.}
\label{wellconnected}
\end{figure}

Note that $S_2$ minors in full components are not necessarily well connected.

\begin{theorem} \label{s2creation} Let $G$ be a 3-connected planar $C$- and $H$-free graph with a three vertex cut $K = \{v_1,v_2,v_3\}$. Suppose this cut creates component $G_1$ such that each of $v_1$, $v_2$, and $v_3$ is adjacent to at least two vertices in $V(G_1)$. Then full component $C_1$ has a well connected $S_2$ minor and a two vertex cut that includes vertex $v_1$ and separates vertices $v_2$ and $v_3$.\end{theorem}

\begin{proof} As $G$ is three-vertex connected, there is a planar embedding such that, imagining a disc such that only vertices $v_1$, $v_2$, and $v_3$ lie on the edge of this disc, all vertices in $V(G_1)$ are contained inside this disc, while all vertices in $V(G) - V(C_1)$ are outside of the disc. By construction, there will be three faces in this embedding that contain a pair of the cut set vertices $v_i$ and $v_j$ ($i \neq j$) and vertices in $V(G_1)$. Then, these faces define three internally disjoint paths $P_1$, $P_2$, and $P_3$, pairwise between $v_1$, $v_2$, and $v_3$, as in Figure \ref{inside}. To prove that these paths are internally disjoint, suppose towards a contradiction that any pair of paths have a vertex $v \in V(G)$ in common. Let $v_i$ be an end vertex for both of these paths. By construction, then, vertices $v$ and $v_i$ are on two common faces. Since we are assuming that $G$ is simple, it follows that $v_i$ and $v$ are either connected by an edge, which contradicts the fact that vertex $v_i$ is assumed to be adjacent to at least two vertices in $V(G_1)$, or $v_i$ and $v$ form a two vertex cut, which contradicts 3-connectivity. 

\begin{figure}[h]
  \centering
    \includegraphics[scale=1.0]{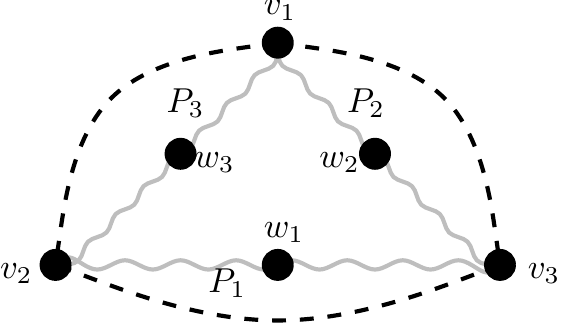}
  \caption{Inside the full component.}
  \label{inside}
\end{figure}

By 3-connectivity, for each path $P_{i}$, there must be a vertex $w_i \in V(P_{i}) \cap V(G_1)$ such that $w_i$ is adjacent to a vertex not in $V(P_{i})$, as otherwise the end vertices of path $P_i$ form a two vertex cut. As such, for each $w_i$, there must be a path in $C_1$ to a vertex in $V(P_j)-V(P_i)$, $i \neq j$. Let $P^1$, $P^2$, and $P^3$ be fixed paths from vertices $w_1$, $w_2$, and $w_3$, respectively, that meet these requirements. To keep these paths as simple as possible, we may assume without loss of generality that $V(P^i) \cap V(P_i) = \{w_i\}$ for each path, and further that $|V(P^i) \cap (V(P_1) \cup V(P_2) \cup V(P_3))| = 2$.  Since $C_1$ is a full component created by a three vertex cut, there is at least one vertex in $V(G) - V(C_1)$ with vertex disjoint paths to each of $v_1$, $v_2$, and $v_3$ in $V(G) - V(G_1)$ by Corollary \ref{mengerscor}. Note that we are trying to avoid minors as shown in Figure \ref{cubek33}. To avoid a minor isomorphic to $H$, one of the paths $P_i$ must intersect all paths $P^j$, so $V(P^j) \cap V(P_i) \neq \emptyset$ for $1 \leq j \leq 3$. Without loss of generality, suppose this is path $P_1$. 

\begin{figure}[h]
  \centering
    \includegraphics[scale=0.75]{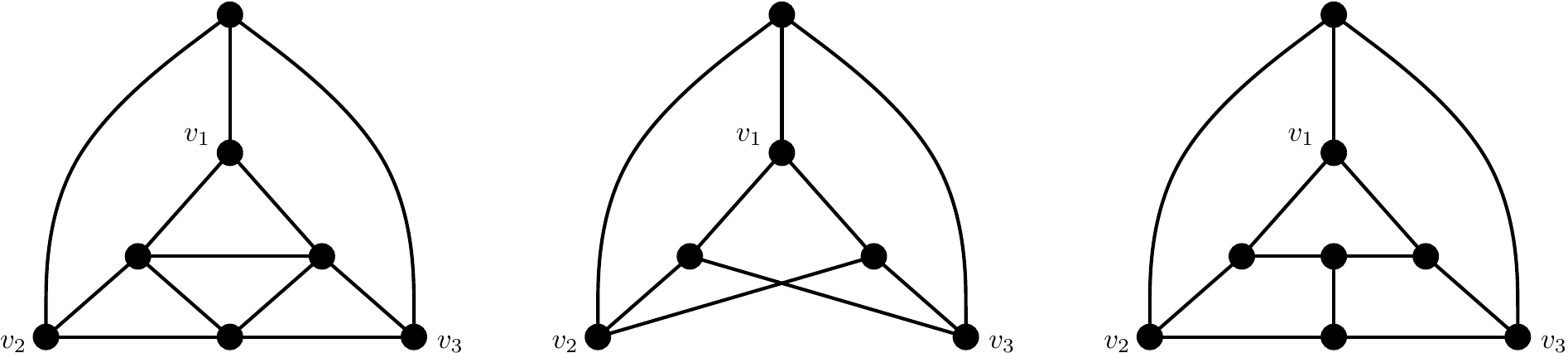}
  \caption{Possible $H$, $K_{3,3}$, and $C$ minors.}
  \label{cubek33}
\end{figure}

Suppose first that path $P^j$ has end vertices $w_j$ and $v_j$, $j \in \{2,3\}$. By symmetry suppose it is path $P^2$. Drawn as in Figure \ref{insideset1}, path $P^3$ must share an internal vertex with $P^2$ to maintain planarity, and have end vertices $v_3$ and $w_3$ to avoid a $C$ minor. As this is a planar embedding, path $P^1$ must therefore contain a vertex in $(V(P^2) \cup V(P^3)) - (V(P_1) \cup V(P_2) \cup V(P_3))$. This creates a $C$ minor.

\begin{figure}[h]
  \centering
    \includegraphics[scale=1.0]{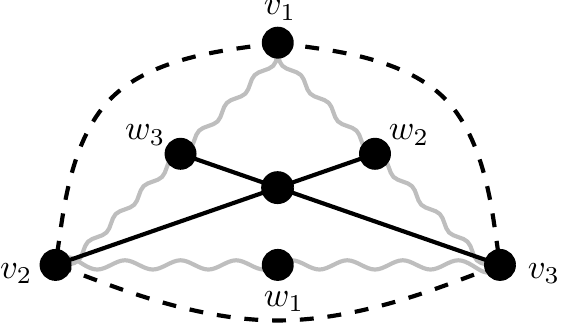}
  \caption{A crossing created by $v_2-w_2$ and $v_3-w_3$ paths.}
  \label{insideset1}
\end{figure}

Whence, we may assume that end vertices for paths $P^2$ and $P^3$ are in $V(G_1) \cap V(P_1)$. To avoid an $H$ minor, there may be no paths from $V(P_2) \cap V(G_1)$ to $V(P_3) \cap V(G_1)$ that do not go through either $P_1$ or $v_1$. Let vertex $w$ be the vertex furthest from $v_3$ on path $P_1$ such that there is no path in $V(G_1)$ from a vertex in $(V(P_3) \cap V(G_1)) - V(P_1)$ to a vertex closer to $v_3$ on $P_1$. By construction, there is at least one such path, and hence the choice of $w$ is well-defined. I claim vertices $v_1$ and $w$ form a two vertex cut in $C_1$.

\begin{figure}[h]
  \centering
    \includegraphics[scale=1.0]{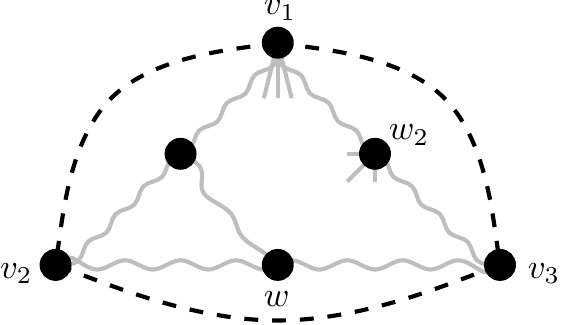}
  \caption{The appearance of vertex $w$ in $C_1$.}
  \label{insideset2}
\end{figure}

By construction, $C_1$ must appear as in Figure \ref{insideset2}. Consider a path from any vertex in $V(P_2) \cap V(G_1)$ to a vertex in $V(P_1) \cap V(G_1)$. If this path is to a vertex further from $v_3$ on $P_1$ than $w$, we create either a $K_{3,3}$ or $C$ minor. As such, all such paths must be either to $w$ or to vertices closer to $v_3$ than $w$ on $P_1$. Thus, vertices $v_1$ and $w$ form a two vertex cut in $C$, and $C$ contains a well connected $S_2$ minor. \end{proof}

\begin{theorem} \label{canthave3} Consider a planar 3-connected $C$- and $H$-free graph $G$. Fix a non-splitting 5-configuration $S \subseteq E(G)$. If a three vertex cut $K = \{v_1,v_2,v_3\}$ produces a full component $C_1$ that contains no well connected $S_2$ minor, then $|S \cap E(C_1)| \neq 3$.  \end{theorem}

\begin{proof} Suppose towards a contradiction there are three edges in $S \cap E(C_1)$. By Theorem \ref{s2creation}, one of the vertices in $K$ must be incident with at most one internal vertex of $C_1$. Without loss of generality, suppose this is vertex $v_1$. Consider the graph $G$, labelled as in Figure \ref{s2free}.

\begin{figure}[h]
  \centering
    \includegraphics[scale=1.0]{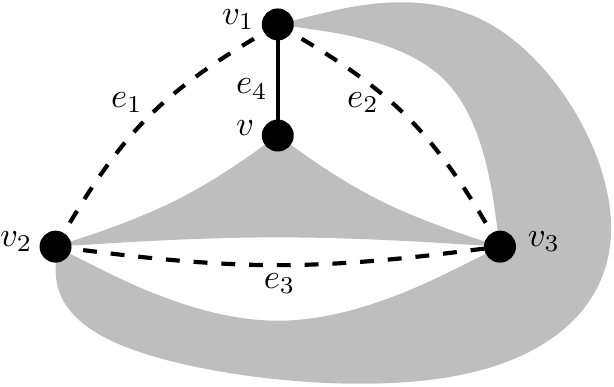}
  \caption{An $S_2$-free full component.}
  \label{s2free}
\end{figure}

From Corollary \ref{trickythreecut}, it is sufficient to prove that one of the edges in the 5-configuration $S$ in $C_1$, when deleted or contracted, creates a two vertex cut that separates the remaining edges of the 5-configuration into pairs in each full component. As such, if at least one of of $e_1$, $e_2$, or $e_3$ is in $S$, contraction of this edge produces a two vertex cut that separates the other two edges. If none of $e_1$, $e_2$, or $e_3$ is in $S$ and $e_4$ is, then deleting $e_4$ produces a two vertex cut that separates the remaining edges in pairs.

If none of $e_1$, $e_2$, $e_3$, or $e_4$ are in $S$, then take the three vertex cut $\{v,v_2,v_3\}$. Since $C_1$ had no well connected $S_2$ minor, this smaller three vertex cut contains no well connected $S_2$ minor and so by Theorem \ref{s2creation} we may continue this process. It becomes impossible to iterate further when there is only a single vertex remaining inside the three vertex cut, as in Figure \ref{s2freelaststep}. If an edge not incident with $w$ is in $S$, contracting this edge produces a two vertex cut separating the remaining edges. If none of these edges are in $S$,  all three edges incident to vertex $w$ must be in $S$, and this 5-configuration will split by Proposition \ref{cycle}. In any case, a 5-configuration with three edges in $E(C_1)$ splits in $G$.

\begin{figure}[h]
  \centering
    \includegraphics[scale=1.0]{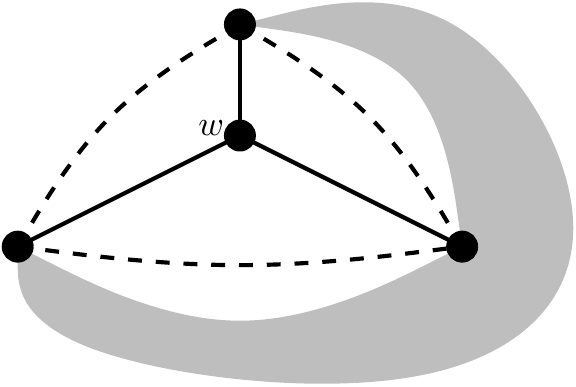}
  \caption{A trivial $S_2$-free full component.}
  \label{s2freelaststep}
\end{figure}

 \end{proof}

\begin{corollary} \label{atmost2} With a graph $G$ as in Theorem \ref{canthave3}, a full component with no well connected $S_2$ minor may not contain more than two edges in a non-splitting 5-configuration $S$. \end{corollary}

\begin{proof} From Theorem \ref{canthave3}, a full component with no well connected $S_2$ minor may not contain three edges of a non-splitting 5-configuration. Suppose then that this full component contains at least four edges of a non-splitting 5-configuration. Take the smallest three vertex cut in this full component such that any smaller three vertex cut properly contained in this full component contains fewer than three edges in $S$,  $C_j$. If $C_j$ contains three edges in the 5-configuration, then $S$ cannot split by Theorem \ref{canthave3}. 

If $C_j$ contains four or five edges of $S$, it must be that, labelled as in Figure \ref{fourorfive}, some set of edges $e_1$, $e_2$, or $e_3$ are in $S$. Specifically, if five edges of $S$ are $C_j$, then all of $e_1,e_2,e_3 \in S$. If four edges of $S$ are in $C_j$, then either two or three of these edges are in $S$. 

\begin{figure}[h]
  \centering
    \includegraphics[scale=1.0]{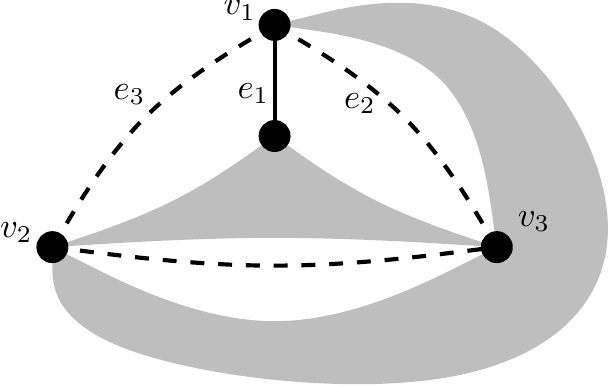}
  \caption{Potential edges in $S$.}
  \label{fourorfive}
\end{figure}

If $e_1 \in S$, then deleting edge $e_1$ produces a two vertex cut at vertices $v_1$ and $v_2$. If this two vertex cut separates the remaining edges of $S$ into sets of two in the full components, then $S$ will split by Corollary \ref{trickythreecut}. This holds if $C_j$ contains five edges of $S$, or if $C_j$ contains four edges of $S$, and one of $e_2$ or $e_3$ is not in $S$. 

If $e_1 \notin S$, it must be the case that $e_2, e_3 \in S$ and precisely four of the edges in $S$ are in this full component. Then, contracting edge $e_2$ edges produces a two vertex cut such that $e_3$ is incident with both vertices of this cut and there is at least one other edge in both full components. One of the full components induced by this two vertex cut must therefore contain precisely two edges in $S$, and by Corollary \ref{trickythreecut}, $S$ does not split in this case.

Suppose now that $C_j$ contains four edges of $S$ and $e_1,e_2,e_3 \in S$. Again, we may contract edge $e_2$, and this produces a two vertex cut. Then, edges $e_1$, $e_3$, and the forth edge in $S \cap E(C_j)$ are in a full component induced by this two vertex cut. Let $f$ be the edge in $S$ but not in $E(C_j)$. Then, a full component induced by this two vertex cut contains precisely edges $e_3, f \in S$, and by Corollary \ref{trickythreecut} $S$ does not split.

In any case, it is not possible that $S$ is a non-splitting 5-configuration. \end{proof}

\begin{corollary} \label{onlys1} Suppose $G$ is 3-connected, planar, and $H$-, $O$-, and $C$-free graph. If $G$ has a three vertex cut such that neither full component has a well connected $S_2$ minor, then $G$ must split. \end{corollary}

\begin{proof} Consider a three vertex cut in $G$ as above. By assumption, both full components have no well connected $S_2$ minor. It follows from Corollary \ref{atmost2} that both full components may contain at most two edges of a non-splitting 5-configuration. Hence, any set of five edges in $E(G)$ will force more than two edges in at least one full component, and as such will split. \end{proof}

\begin{lem} \label{fan} Let $G$ be a 3-connected $H$-free graph such that there is a three vertex cut that produces a full component $C_1$ isomorphic to $S_2$, and this $S_2$ is well connected. If there is a cycle $L$ in $G$ such that $E(C_1) \cap E(L) = \emptyset$ and $v_1 \notin V(L)$ then $G$ has an $H$ minor. As a result, the vertices external to $C_1$ must appear as in Figure \ref{fangraph}.  \end{lem}

\begin{figure}[h]
  \centering
    \includegraphics[scale=1.0]{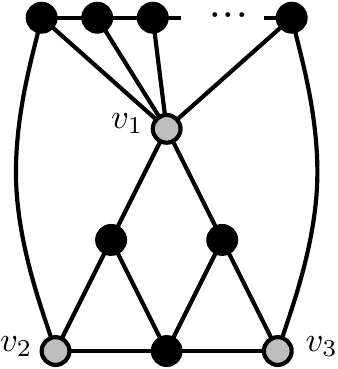}
  \caption{A graph containing a well connected $S_2$ subgraph.}
\label{fangraph}
\end{figure}

Note that the number of vertices external to the $S_2$ full component is simply required to be at least one.

\begin{proof} Suppose there is a cycle $L$ in $G$ that shares no edges with $E(C_1)$. Suppose vertices $v_1$, $v_2$, and $v_3$ are not in $V(L)$. Taking vertices $w_1,w_2,w_3 \in V(L)$ and any vertex $v \in V(C_1) -\{v_1,v_2,v_3\}$, by Corollary \ref{mengerscor}, there must be internally vertex disjoint paths from $v$ to each of $w_1$, $w_2$, and $w_3$. By construction, each of $v_1$, $v_2$, and $v_3$ must be contained in precisely one of these paths, and each path must contain precisely one of these vertices. Up to relabeling, this will create an $H$ minor as in Figure \ref{htrisub}. Similarly, if the cycle includes $v_2$, $v_3$, or both $v_2$ and $v_3$, but not $v_1$, the graph will an $H$ minor. As such, any cycle that contains no edges in $E(C_1)$ must contain vertex $v_1$.

\begin{figure}[h]
  \centering
    \includegraphics[scale=1.0]{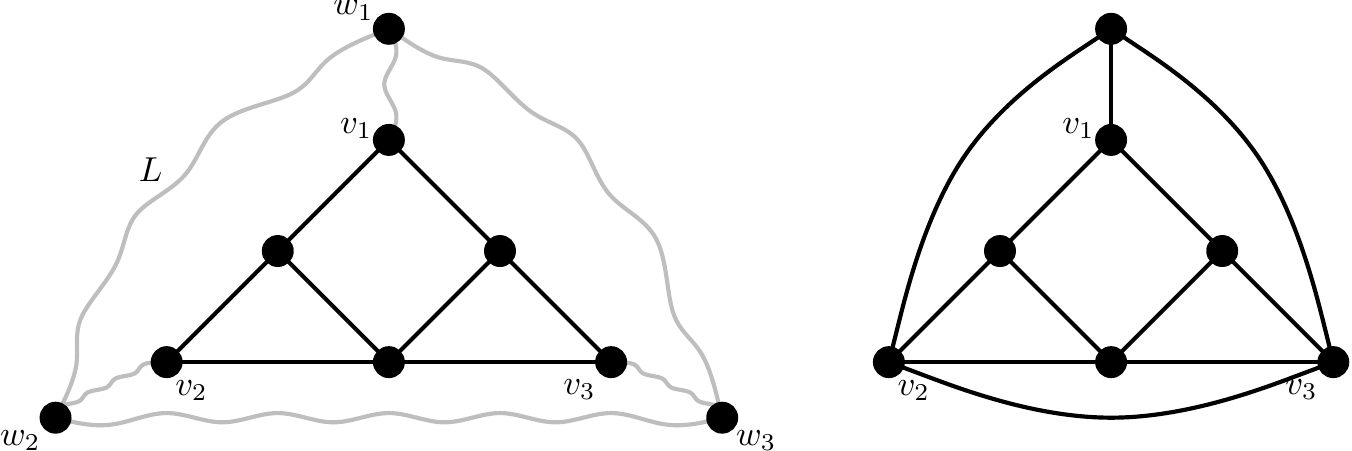}
  \caption{The cycle $L$ in $G$ and the $H$ minor it creates.}
\label{htrisub}
\end{figure} 

Now suppose $v$ is a vertex in $V(G) - V(C_1)$. By definition of a vertex cut set for a non-complete graph, such a vertex must exist. Again by Menger's Theorem, there must be vertex disjoint paths from $v$ to each of $v_1$, $v_2$, and $v_3$ outside of $C_1$. If there is a vertex $v'$, $v' \neq v$, external to $C_1$, then $v'$ must lie on one of these paths, as otherwise $G$ would have a $K_{3,3}$ minor or a cycle that shares no edges with $C_1$ and does not contain $v_1$. Hence we may assume all vertices in $V(G) - V(C_1)$ lie on these paths. By the same reasoning, $v$ must lie on a path from $v'$ to one of the cut vertices. Since we may not create a cycle external to $C_1$ that does not contain $v_1$, it is not possible that $v$ lies on the path from $v'$ to $v_1$, and similarly that $v'$ lies on the path from $v$ to $v_1$. That is, if such a path is non-trivial, 3-connectivity forces the vertex incident with $v_1$ on this path to be adjacent to at least one vertex in addition to its neighbours on the path. This will create a cycle that does not include $v_1$. Hence, the external vertices of $g$ must appear as in Figure \ref{fangraph}. \end{proof}

\begin{corollary} \label{fanminor} Let $G$ be a 3-connected graph with a three vertex cut $\{v_1,v_2,v_3\}$ that induces a full component $C_1$. If $C_1$ has a well connected $S_2$ minor, the vertices in $V(G) - V(C_1)$ must appear as in Figure \ref{fangraph}. Similarly, with $C_1$ drawn as in Figure \ref{wellconnected}, every cycle that contains no edges in $E(C_1)$ must contain vertex $v_1$. \end{corollary}

This is an immediate consequence of the $H$ minor that will otherwise appear.

\begin{definition} Suppose a graph $G$ has a full component with a well connected $S_2$ minor. We call $G$ a \emph{fan graph}, and the vertices (edges) external to this full component \emph{fan vertices} (\emph{fan edges}). The set of fan edges and vertices create the \emph{fan}. \end{definition}

\begin{theorem} Let $G$ be a simple planar 3-connected graph with a non-splitting 5-configuration. Then, $G$ must have an $H$, $O$, or $C$ minor. \end{theorem}

\begin{proof} Since $G$ is assumed to be 3-connected, planar, and $O$-free, by Theorem \ref{notfour} it must have a three vertex cut. If there is a three vertex cut such that neither full component contains a well connected $S_2$ minor, then by Corollary \ref{onlys1} the graph will split. Suppose then that a full component has a well connected $S_2$ minor. Take the three vertex cut that creates the smallest full component $C_1$ containing a well connected $S_2$ minor. Then, each of $v_1$, $v_2$, and $v_3$ must be incident with two internal vertices of $C_1$, as otherwise this contradicts the minimality condition of this three vertex cut. By Corollary \ref{fanminor}, $G$ is a fan graph. Label $G$ as in Figure \ref{finalg}.

\begin{figure}[h]
  \centering
    \includegraphics[scale=1.0]{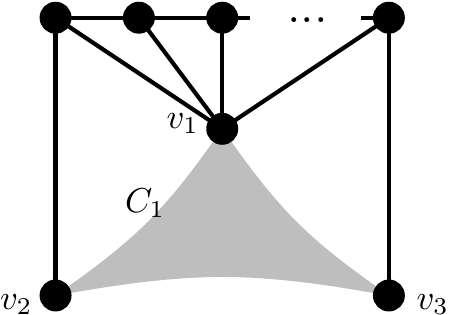}
  \caption{The known vertices of $G$. Note that the fan may have only a single vertex.}
\label{finalg}
\end{figure}

By Theorem \ref{s2creation}, there is a two vertex cut in $C_1$. Let $f$ be an arbitrary fan vertex. Label this as in Figure \ref{finalgbonus}. Note that there is a three vertex cut $V = \{v,v_1,f\}$ in $G$.

\begin{figure}[h]
  \centering
    \includegraphics[scale=1.0]{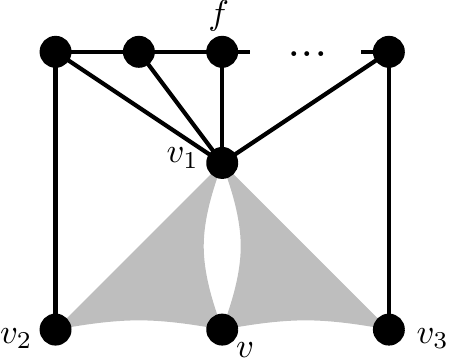}
  \caption{Graph $G$, with known vertices labelled.}
\label{finalgbonus}
\end{figure}

Again, if neither full component induced by vertex cut set $V$ contains a well connected $S_2$ minor then the graph splits by Corollary \ref{onlys1}. Suppose then that one full component has a well connected $S_2$ minor. Consider the full component created by vertex cut set $V$ that does not contain the well connected $S_2$ minor. By construction, it must contain a minor as in Figure \ref{finalgnewsplit}.

\begin{figure}[h]
  \centering
    \includegraphics[scale=1.0]{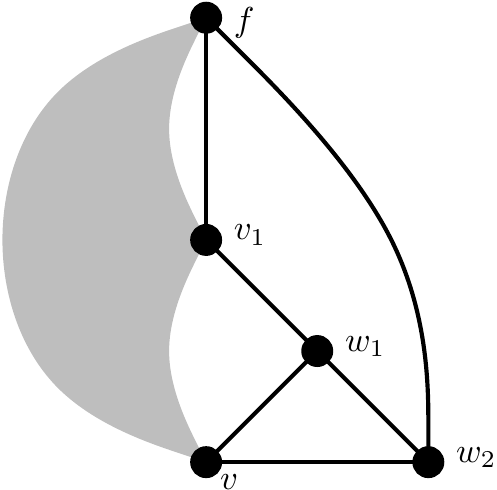}
  \caption{Cycles outside of a full component assumed to have a well connected $S_2$ minor.}
\label{finalgnewsplit}
\end{figure}

Specifically, vertex sets $V_1 = \{f,v_1,w_1,w_2\}$ and $V_2 = \{v,w_1,w_2\}$ induce cycles. Note that $V \cap V_1 = \{f,v_1\}$ and $V \cap V_2 = \{v\}$. As a minor of our graph $G$, these cycles necessarily exist in $G$, connected to $f$, $v$, and $v_1$ as in this minor. No matter how the $S_2$ subgraph appears, then, there is a cycle that will create an $H$ minor by Corollary \ref{fanminor}. Hence, no 3-connected, simple, planar, $H$-, $O$-, and $C$-free graph has a non-splitting 5-configuration.  \end{proof}

It follows from this theorem that a simple, 3-connected, planar, $H$-, $O$-, and $C$-free graph must have a three vertex cut such that neither full component created by this cut has a well connected $S_2$ minor. Specifically, if a three vertex cut creates a full component that does have a well connected $S_2$ minor, we may use it to find a three vertex cut that does not create a full component with a well connected $S_2$ minor. It follows from Theorem \ref{s2creation} that, for either connected component created by this cut, there is a vertex in this cut set incident with precisely one vertex the connected component. As in Theorem \ref{canthave3}, we may replace the vertex in the three vertex cut with this vertex in the connected component, each time producing a new three vertex cut in our original graph until there is a single vertex in the connected component. We may use this technique to characterize the 3-connected simple graphs that split.

\begin{corollary} Suppose $G$ is a 3-connected simple graph that splits and $|E(G)| = n$. There is an ordering on the edges $e_1, e_2, ... , e_n$ such that, defining $G_i$ to be the graph induced by edges $\{e_1, ... , e_i\}$, $G_i$ is a subgraph of $G_{i+1}$ for $1 \leq i \leq n-1$ and there are precisely three vertices in the intersection of $G_i$ and $G \cut G_i$ for $3 \leq i \leq n-3$. \end{corollary}

\begin{proof} As we know such a cut exists, as mentioned prior, consider a three vertex cut in $G$ such that neither full component created by this cut has a well connected $S_2$ minor. Label these full components $C_1$ and $C_2$. As previously stated, we may decompose these to $K_{1,3}$ by finding three vertex cuts in the original graph that contain two of the vertices in the previous two vertex cut. We create an ordered list of the edges using this decomposition. Specifically, suppose vertices $v_1$, $v_2$, and $v_3$ are the vertices in the cut set at some stage of this decomposition, and $v_1$ is incident with only one vertex in the connected component being considered (such a vertex exists by Theorem \ref{s2creation}). Add edges incident with two of these vertices (that are not already in the list) to the queue in any order, then add edge incident with $v_1$ and a vertex in the connected component, before moving in to the smaller three vertex cut. When you reach the last step of the decomposition, the $K_{1,3}$, add the remaining three edges in any order. This creates an ordering on the edges of $C_1$ and $C_2$. Reversing the order for edges in $E(C_1)$ and concatenating the ordered edges in $E(C_2) - E(C_1)$ produces an ordering of the edges in $E(G)$.

The corollary follows immediately from this ordering. By construction, for $3 \leq i \leq n-3$, both $G_i$ and $G \cut G_i$ are subgraphs of full components created by a three vertex cut, and as such there are precisely three vertices in the intersection of these graphs.\end{proof}

%% file: files/last.tex
%
%

\chapter{Conclusions and Open Problems}
\label{last}

\section{The Results}

The main result of this thesis is that, if a simple 3-connected graph splits, it must be $H$-, $O$-, $C$-, $K_5$-, and $K_{3,3}$-free. This was achieved by finding restrictions to the distribution of 5-configurations across vertex cut sets, and specifically the number of edges in a non-splitting 5-configuration that may appear in full components.

Recall that the method of denominator reduction given in Chapter \ref{introduction} is used only for primitive divergent graphs. It is important to show then that there are primitive divergent $H$-, $O$-, $C$-, $K_5$-, and $K_{3,3}$-free primitive graphs. Consider the graph in Figure \ref{zigzagarb}. We call this a \emph{zigzag graph}. This is an infinite class of graphs, the smallest being $K_4$, drawn in Figure \ref{zigzagsmall} to emphasize the appearance.

\begin{figure}[h]
  \centering
      \includegraphics[scale=1.0]{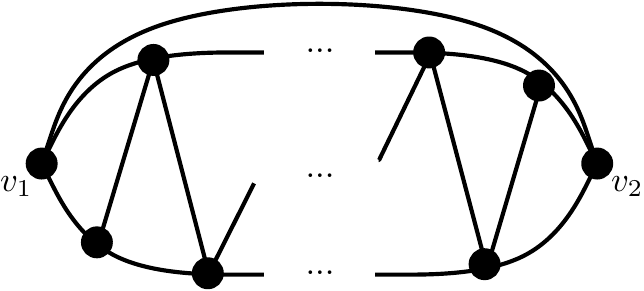}
  \caption{A zigzag graph of arbitrary size. Note that zigzag graphs with order of either parity is allowed.}
\label{zigzagarb}
\end{figure}

\begin{figure}[h]
  \centering
      \includegraphics[scale=1.0]{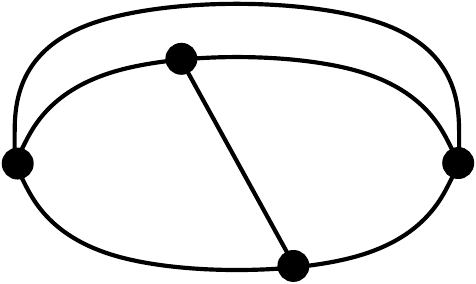}
  \caption{The smallest possible zigzag graph.}
\label{zigzagsmall}
\end{figure}

From Proposition \ref{small}, $K_4$ splits. Consider now a zigzag graph $G$ such that $|V(G)| > 4$. Labelled as in Figure \ref{zigzagarb}, every three vertex cut must contain precisely one of $v_1$ or $v_2$. Then, there is no three vertex cut that produces a connected component $G_1$ and full component $C_1$ such that each vertex in $V(C_1) - V(G_1)$ is incident with at least two vertices in $V(G_1)$. By Theorem \ref{s2creation} a zigzag graph cannot have a well connected $S_2$ minor, and hence by Corollary \ref{onlys1} $G$ splits.

Showing that zigzag graphs are primitive divergent is more involved. Recall that for a primitive divergent graph $G$ and any proper subgraph $\gamma$ with at least one edge in every connected component must have $|E(\gamma)| > 2h_\gamma$ where $h_\gamma = |E(\gamma)| - |V(\gamma)| + k_\gamma$. Simplifying, \begin{align*} |E(\gamma)| &> 2(|E(\gamma)| - |V(\gamma)| + k_\gamma) \\  -|E(\gamma)| &>  - 2|V(\gamma)| + 2k_\gamma \\ |E(\gamma)| &< 2|V(\gamma)| - 2k_\gamma. \end{align*}

Note that in general we need to consider only connected subgraphs $\gamma$. If this inequality does not hold for a disconnected subgraph, then it does not hold for at least one of the connected components in this subgraph.

\begin{prop} \label{zigzag} Zigzag graphs are primitive divergent. \end{prop}

\begin{proof} Suppose $G$ is a zigzag graph and $G$ has $n$ vertices of degree four. By construction, $G$ has four vertices with degree three and is connected. Then, $|E(G)| = 2n +6$, and $h_G = (2n+6) - (n+4) +1 = n+3$. Thus, $|E(G)| = 2h_G$. 

Suppose that $\gamma$ is a connected non-trivial proper subgraph of $G$. We may assume that $\gamma$ is induced by its vertex set, as if the inequality holds in this case it will hold for subgraphs with the same vertex set but fewer edges. By construction, this inequality holds so long as the sum of the degrees of vertices in $V(\gamma)$ is less than $4|V(\gamma)| - 4$. Given that the maximum vertex degree in $G$ is four, this would mean that adding at most two edges (possibly loops or parallel edges) would make $\gamma$ 4-regular. Since $\gamma$ must be a proper subgraph, this is impossible.

Thus, $G$ is a primitive divergent graph.  \end{proof}

It follows from Proposition \ref{zigzag} that there are an infinite number of graphs that are both primitive divergent and split. There are graphs outside of the zigzag family that are primitive divergent and split, so this is not a complete set.

The complete set of forbidden minors for splitting graphs remains unsolved. Specifically, we have not necessarily found all minor-minimal non-simple or 2-connected graphs with non-splitting 5-configurations. While a number of restrictions were found in Chapter \ref{cuts} to the appearance of minor-minimal non-splitting graph with a two vertex cut, these do not provide a clear method of fully resolving this case.

These results, however, are satisfactory for approaches to partial Feynman integrals using five-invariants. Recall that Propositions \ref{52}, \ref{129}, and \ref{130}, from which the value of splitting graphs is derived, are all stated in terms of primitive divergent graphs.

\begin{prop} \label{noparalleledges} If a graph $G$ has parallel edges, then it is not primitive divergent. \end{prop}

\begin{proof} Suppose $e_1$ and $e_2$ are a pair of parallel edges in $G$. Then, the subgraph of $G$ induced by edges $e_1$ and $e_2$, call it $\gamma$, has $|V(\gamma)| = 2$, $k_\gamma = 1$, and $|E(\gamma)| = 2$. Then, $$h_\gamma = |E(\gamma)| - |V(\gamma)| + k_\gamma = 1.$$ Hence, $|E(\gamma)| = 2h_\gamma$, and hence graph $G$ is not primitive divergent. \end{proof}

\begin{prop} \label{nodegreetwo} If a graph $G$ has a vertex of degree two, then it is not primitive divergent. \end{prop}

\begin{proof} First suppose that there is an edge $e \in E(G)$ that is a  bridge. Then, the subgraph $\gamma$ of $G$ induced by edge set $E(G) - \{e\}$ has $|V(G)|$ vertices, $|E(G)|-1$ edges, and $k_G +1$ connected components. Thus, $h_\gamma = h_G$. Assuming that $|E(G)| = 2h_G$, we have $$|E(\gamma)| = |E(G)|-1 = 2h_G - 1 = 2h_\gamma - 1.$$ Then, $|E(\gamma)| < h_\gamma$, and as such a graph with a bridge cannot be primitive divergent.

Suppose $v \in V(G)$ has degree two. From the previous paragraph, we may assume that the graph $\gamma$ induced by $V(G)-\{v\}$ has $|V(G)| -1$ vertices, $k_G$ connected components, and $|E(G)| -2$ edges. Then, \begin{align*} h_\gamma &= (|E(G)| -2) - (|V(G)| -1) + k_G \\ &= (|E(G)| -|V(G)| + k_G) -1 \\ &= h_G -1.\end{align*} Assuming then that $|E(G)| = 2h_G$, we find that $$|E(\gamma)| = |E(G)| -2 = 2(h_G-1) = 2h_\gamma,$$ and hence $G$ is not primitive divergent. \end{proof}

Note in particular that Propositions \ref{noparalleledges} and \ref{nodegreetwo} are the simplest possible cases in which subgraphs can fail the checks for primitive divergence. In the case of a graph with parallel edges, a subgraph induced by two edges fails the inequality check. In a graph with a vertex of degree two, the graph induced by all vertices except one fails the inequality check.

By Corollary \ref{howitlooks}, we know that any minor-minimal non-splitting graph with a two vertex cut has either a pair of parallel edges or a vertex of degree two. It follows from Propositions \ref{noparalleledges} and \ref{nodegreetwo} that any further forbidden minors will not be primitive divergent. Further, translating Theorem 2.10 in \cite{schnetz} into more graph theoretic terms, we may split a graph $G$ into two graphs $G_1$ and $G_2$ as in Figure \ref{splittwocut}, and the residue of the Feynman integral of $G$ is the product of the residues of $G_1$ and $G_2$. As such, we may split graphs with two vertex cuts into smaller graphs until all graphs considered are 3-connected or trivial, and as such we need never calculate residues for graphs with two vertex cuts.

\begin{figure}[h]
  \centering
      \includegraphics[scale=1.0]{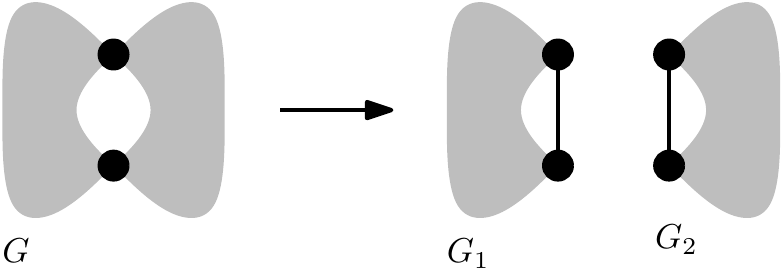}
  \caption{Splitting a graph with a two vertex cut into two graphs.}
\label{splittwocut}
\end{figure}

\section{Open Problems}

While non-simple or 2-connected graphs may not be of particular interest when considering the Feynman integral motivation to our problem, it is still interesting as a graph theoretic problem. I am hesitant to put as a conjecture that the list we have constructed of forbidden minors is complete, but I do consider finding the complete set an appealing open problem.

\begin{question} What is the complete set of obstructions for a graph splitting for all 5-configurations? \end{question}

Further open problems relate to the Feynman integral motivation, especially in light of the previous section.

\begin{question} What is the full set of graphs that are both primitive divergent and splitting? \end{question}

Recall from Proposition \ref{130} that we may calculate $(n+1)^\text{st}$ stage partial Feynman integral denominators $P_{n+1}$ from $P_n$ if $P_n$ is linear in some Schwinger coordinate. The following conjecture is of particular interest, as it would mean the graphs considered never reach a point of obstruction in this computation.

\begin{conjecture} \label{conj1} If a simple 3-connected graph $G$ has no minors isomorphic to $K_5$, $K_{3,3}$, $H$, $O$, or $C$, then there is an ordering of the edges such that, for $n \geq 5$, denominator $P_n$ is factors into terms linear in Schwinger coordinate $\alpha_{e_n}$. That is, simple 3-connected graphs with none of these forbidden minors are completely denominator reducible. \end{conjecture}

A stronger conjecture, though more doubtful, is that any ordering of the edges of the graph in Conjecture \ref{conj1} is denominator reducible.

The demonstrated method that we use to calculate denominators relies on denominators factoring into terms linear for a particular Schwinger coordinate. As such, splitting, which forces the five-invariant to be the product of two Dodgson polynomials, is a particular way in which this may occur. The following conjecture addresses this.

\begin{conjecture} If $G$ is a graph and $S \subset E(G)$, the five invariant $^5\Psi(S)$ factors into terms linear in each Schwinger coordinate if and only if $S$ splits in $G$. \end{conjecture}

Lastly, we prove in Theorem \ref{deltayminimal} in Chapter \ref{delta-y} that if two graphs differ by a $\Delta$-Y transformation, a 5-configuration that contains no edges used in the transformation will either split or not split in both graphs. Logically, we may ask the converse.

\begin{question} \label{conj3} Let $G$ be a non-splitting graph. If $G'$ differs from $G$ by a single $\Delta$-Y transformation, is it the case that $G'$ has a non-splitting 5-configuration if and only if there is a particular non-splitting 5-configuration in $G$ that uses none of the edges in the three valent vertex or induced triangle used to produce $G'$? \end{question}

%% file: files/bibl.tex
%
%

%% file: files/ind.tex
%
%

\renewcommand{\baselinestretch}{\tighttextstretch} 
\normalsize

\addcontentsline{toc}{chapter}{Index}
\typeout{Index}
\printindex

\renewcommand{\baselinestretch}{\textstretch} 
\normalsize